\documentclass[11pt, reqno]{amsart}
\usepackage{mathrsfs}
\usepackage[active]{srcltx}
\usepackage{amsmath,amsthm}    
\usepackage{amssymb}    
\usepackage{mathtools}
\usepackage{longtable}
\usepackage{todonotes}
\usepackage[]{mdframed}
\usepackage[autostyle]{csquotes}
\allowdisplaybreaks

\UseRawInputEncoding

\usepackage{xcolor}
\definecolor{bluecite}{HTML}{0875b7}
\usepackage[unicode=true,
bookmarksopen={true},
pdffitwindow=true,
colorlinks=true,
linkcolor=bluecite,
citecolor=bluecite,
urlcolor=bluecite,
hyperfootnotes=false,
pdfstartview={FitH},
pdfpagemode= UseNone]{hyperref}

\newcommand{\ds}{\displaystyle}
\newcommand{\R}{{\mathbb R}}
\DeclareMathOperator{\Ent}{\mathbf{Ent}}

\newtheorem{proposition}{Proposition}[section]
\newtheorem{theorem}{Theorem}[section]

\newtheorem{corollary}{Corollary}[section]

\newtheorem{remark}{Remark}[section]
\numberwithin{equation}{section}

\usepackage{bbm}  
\usepackage{dsfont}

\author{Zolt\'an M. Balogh and  Alexandru Krist\'aly}	

\address{Mathematisches Institute,
	Universit\"at Bern,
	Sidlerstrasse 5,
	3012 Bern, Switzerland}

\email{zoltan.balogh@unibe.ch}

\address{Department of Economics, Babe\c s-Bolyai University, Str. Teodor Mihali 58-60, 400591 Cluj-Napoca,
	Romania \& Institute of Applied Mathematics, \'Obuda
	University, B\'ecsi \'ut 96/B, 1034
	Budapest, Hungary}

\email{alexandru.kristaly@ubbcluj.ro; kristaly.alexandru@uni-obuda.hu}

\subjclass[]{26D15; 35B35; 47J20; 34K20}

\keywords{Stability; hypercontractivity; logarithmic Sobolev inequality; Pr\'ekopa--Leindler inequality.}
\thanks{Z. M. Balogh is
	supported by the Swiss National Science Foundation, Grant Nr. {200021\_228012}. 
	\thanks{A.\ Krist\'aly is  supported by the
		Excellence Researcher Program \'OE-KP-2-2022 of \'Obuda University, Hungary.} 
}

\title[Sharp stability results in HC and LSI]{Sharp stability in hypercontractivity estimates and
	logarithmic Sobolev inequalities}
\date{} 

\dedicatory{To our friend and colleague, K\'aroly  B\"or\"oczky, with great appreciation}

\begin{document}
	\maketitle

	\begin{abstract}We  prove stability results in  hypercontractivity estimates for the Hopf--Lax semigroup in $\mathbb R^n$ and apply them to deduce stability results for the Euclidean $L^p$-logarithmic Sobolev inequality  for any $p>1$. As a main tool, we use recent stability results for the Pr\'ekopa--Leindler inequality,  due to B\"or\"oczky and De (2021), Figalli and Ramos (2024) and  Figalli, van Hintum, and Tiba (2025). Under mild assumptions on the functions, most of our stability results turn out to be sharp, as they are reflected in the optimal exponent $1/2$  both in the hypercontractivity and $L^p$-logarithmic Sobolev deficits, respectively.  This approach also works for establishing  stability of  Gaussian  hypercontractivity estimates and Gaussian logarithmic Sobolev inequality, respectively. 
	\end{abstract}
	
	\section{Introduction}
	
	An important problem in Geometric Analysis is the characterization of  equality cases	in various geometric and functional inequalities. An even more challenging question is the stability of these inequalities.  Here, we are interested to know, how far the set or the function is from the family of extremizers   (known from the equality case)  when we are close to the equality in the studied inequality. 
	This circle of problems received an increasing attention in the last two decades. In particular, the quantitative stability of the Brunn--Minkowski and isoperimetric inequalities has been studied by Figalli and Jerison \cite{FigalliJerison},   Fusco, Maggi and Pratelli \cite{FuscoMaggiPratelli},   Figalli,  Maggi and Pratelli \cite{FigalliMaggiPratelli} and Figalli, van Hintum and  Tiba \cite{FigallivanHintumTiba_23, FigallivanHintumTiba_24}. 
	
	The functional version of the Brunn--Minkowski inequality is the 
	\textit{Pr\'ekopa--Leindler inequality} (which is a particular form of the Borell--Brascamp--Lieb inequality),  stating that if $u, v, w: \mathbb{R}^n \rightarrow \mathbb{R}_{+}$ are integrable functions and $\lambda\in (0,1)$ such that
	$$
	w(\lambda x+(1-\lambda) y) \geq u(x)^\lambda v(y)^{1-\lambda} \ \ {\rm for\ every}\ \ x, y \in \mathbb{R}^n, 
	$$
	then  
	$$\ds\int_{\mathbb{R}^n} w\mathrm{d} x  \geq  \left(\int_{\mathbb{R}^n} u \mathrm{d} x \right)^\lambda \left(\int_{\mathbb{R}^n} v \mathrm{d} x \right)^{1-\lambda} .$$
	Equality holds in the latter inequality if and only if the three functions $u,v$ and $w$ (up to dilations and translations) are equal to a log-concave function . 
	The stability of the Pr\'ekopa--Leindler inequality (and also in Borell--Brascamp--Lieb inequality) has been investigated e.g.\ by Ball and   B\"or\"oczky \cite{BallBoroczky1, BallBoroczky2},  B\"or\"oczky,    Figalli and Ramos \cite{Boroczky-Figalli-Ramos}, 
 Bucur and   Fragal\`a \cite{BucurFragala}, Figalli, van Hintum and  Tiba \cite{FigallivanHintumTiba_25}, Figalli and Ramos \cite{FigalliRamos} and Rossi and  Salani \cite{RossiSalani}. In this approach one can reduce the  problem to the  stability of the Brunn--Minkowski inequality for level sets of the functions. In particular, the almost equality in the Pr\'ekopa--Leindler inequality shows that the three functions $u,v$ and $w$ (up to dilations and translations) are close in a some sense to a log-concave function.
	
	Extensive studies concern the stability of $L^p$-Sobolev-type inequalities as well. Starting from the seminal work of  Bianchi and  Egnell \cite{BianchiEgnell}, sharp quantitative results for the classical Sobolev inequality have been established e.g.\ by Deng,  Sun and Wei \cite{DSW}, Figalli and   Zhang \cite{FigalliRu},  Dolbeault,   Esteban, Figalli, Frank and Loss \cite{Dolbeault-et-al}. We note  that these results appear to be sharp only in the case $p=2$. 
	
	Closely related to this topic is the stability of logarithmic Sobolev inequalities, see e.g.\   Bez,  Nakamura and Tsuji \cite{BNT},  Brigati,  Dolbeault and Simonov \cite{BDS}, Dolbeault,   Esteban, Figalli, Frank and Loss \cite{Dolbeault-et-al, Dolbeault-et-al_short}, Fathi, Indrei and Ledoux \cite{FathiIndreiLedoux}, Feo,  Indrei,  Posteraro and  Roberto \cite{FeoIndreiPosteraroRoberto}, Indrei and Marcon \cite{Indrei-Marcon}, Suguro \cite{Takshi}.  We notice that these papers address the stability of $L^2$-logarithmic Sobolev inequalities, mostly with respect to the Gaussian measure.  
	
	In view of the aforementioned results, one of the main goals of this paper is 
	to establish stability results for the \textit{$L^p$-Euclidean logarithmic Sobolev inequality} for all $p>1$ (with the Euclidean instead of the Gaussian measure), which will be derived by \textit{hypercontractivity estimates} for the Hopf--Lax semigroup.
	In fact, the strategy is to follow the chain of implications:
	
	%

	
	
	\begin{center}\label{Sys}
		\hspace{1cm}	\fbox{\parbox{3cm}{Stability in \\ Pr\'ekopa--Leindler inequality }} $\implies$ \fbox{\parbox{3.1cm}{Stability in \\ Hypercontractivity estimate}} $\implies$ \fbox{\parbox{3.5cm}{Stability in \\Logarithmic Sobolev inequality }} \hspace{2cm} \textbf{(S)}
	\end{center}
	
	
	To our knowledge, the first implication and the corresponding stability of the hypercontractivity estimate are new. Moreover, they hold under minimal assumption on the class of functions to be considered. Although results on  the stability of the  logarithmic Sobolev inequality are available, our method using the above chain of implications is different than the ones used up until now. In addition, it allows us to handle the class of $L^p$-Euclidean logarithmic Sobolev inequalities instead of the Gaussian logarithmic Sobolev inequalites that has been considered before with other methods.

	We recall that the $L^p$-Euclidean logarithmic Sobolev inequality  and the hypercontractivity estimate  of the Hopf--Lax semigroup are equivalent, as shown in the works of    Gentil \cite{Gentil_JFA, Gentil_BullSci} and  Bobkov, Gentil and Ledoux \cite{BGL}. According to the above scheme \textbf{(S)},  the stability of the $L^p$-Euclidean logarithmic Sobolev inequality  will be derived by the stability of the hypercontractivity estimate for the Hopf--Lax semigroup. Therefore, we start with the main properties of the Hopf--Lax semigroup and state our stability results with respect to its hypercontractivity.

	
	Given an (enough regular) function $g:\R^n\to \mathbb R$, let $(\mathbf{Q}_t)_{t \geq 0}$ be the family of operators defined by 
	$$ 
	\left\{
	\begin{array}{l}
		\mathbf{Q}_t g(x)=\inf _{y \in \R^n}\left\{g(y)+\frac{|x-y|^{p^{\prime}}}{p^{\prime} t^{p^{\prime}-1}}\right\}, \quad t>0, \quad x  \in \mathbb{R}^n, \\
		\mathbf{Q}_0 g(x)=g(x), \quad x  \in \mathbb{R}^n.
	\end{array}
	\right.
	$$
	The family of operators $(\mathbf{Q}_{t})_{t \geq 0}$ defines a nonlinear semigroup, called the Hopf--Lax semigroup, which, according to Gentil \cite{Gentil_JFA}, satisfies a sharp hypercontractivity estimate. 
	More precisely, if  $0<\alpha< \beta$, then for any $t>0$ and $g:\mathbb R^n\to \mathbb R$ with the property that if $e^g\in L^\alpha(\mathbb R^n)$, then  $e^{\mathbf{Q}_t g}\in L^\beta(\mathbb R^n)$, and we have the estimate
	\begin{equation} \label{HCI_ineq}
		\left\|e^{\mathbf{Q}_{t} g}\right\|_{\beta} \leq  \mathcal{C}_{p, t, \alpha, \beta} \cdot \left\|e^g\right\|_{\alpha} ,
	\end{equation}
	where the constant 
	\begin{equation} \label{HCI_constant}
		\mathcal{C}_{p, t, \alpha, \beta} = \left(\frac{\beta-\alpha}{t}\right)^{\frac{n}{p} \frac{\beta-\alpha}{\alpha \beta}} \frac{\alpha^{\frac{n}{\alpha \beta}\left(\frac{\alpha}{p}+\frac{\beta}{p^{\prime}}\right)}}{\beta^{\frac{n}{\alpha \beta}\left(\frac{\beta}{p}+\frac{\alpha}{p^{\prime}}\right)}}\left(\left(p^{\prime}\right)^{\frac{n}{p^{\prime}}} \Gamma\left(\frac{n}{p^{\prime}}+1\right) \omega_n\right)^{\frac{\alpha-\beta}{\alpha \beta}}
	\end{equation}
	is optimal. Hereafter, $\omega_n$ stands for the volume of the unit ball in $\mathbb R^n$ and, for simplicity, $\|\cdot\|_\alpha$ denotes  the usual $L^\alpha$-norm in $\mathbb R^n$ with the standard  Lebesgue measure.  In addition, equality holds in \eqref{HCI_ineq} if and only if
	\begin{equation} \label{HCI_extremizers}
		g(x) = C - \left( \frac{\beta - \alpha}{\beta t} \right)^{p'-1}  \frac{\left| x - x_0 \right|^{p'}}{p'} , \ x\in \mathbb R^n,
	\end{equation}
	for some $x_0 \in \R^n$ and $C \in \R$, see Balogh, Don and Krist\'aly \cite{BDK}.

	Let us define the \textit{hypercontractivity deficit} associated with the estimate \eqref{HCI_ineq} as
	\begin{equation}  \label{HCI_deficit}
		\delta^{\sf HC}_{p ,t,\alpha,\beta}(g) = \left(\mathcal{C}_{p, t, \alpha, \beta}\right)^\alpha \frac{
			\left\|e^g\right\|^\alpha_{\alpha} }{\left\|e^{\mathbf{Q}_{t} g}\right\|_{\beta}^\alpha} -1  \geq 0,
	\end{equation}
	where  $\mathcal{C}_{p, t, \alpha, \beta}$ is the optimal constant from \eqref{HCI_constant}. 
	
	We can now state our first main result, that is concerning the stability of the hypercontractivity inequality \eqref{HCI_ineq} under minimal (integrability) assumptions on the functions. 
	
	\begin{theorem}\label{main-theorem-HC} Let $n \geq 1$, $p>1$  and   $0<\alpha< \beta$. Then, for any $t>0$, there exists a constant $ C(n,p,  \alpha,\beta,t) >0$ with the following property. Given any function $g:\mathbb R^n\to \mathbb R$ such that $e^g\in L^\alpha(\mathbb R^n)$, 
		there exists a point $x_0 \in \mathbb{R}^n$ such that
		\begin{equation}\label{stab-HC}
			\int_{\R^n}\Big|e^{-\theta \cdot \frac{|x-x_0|^{p^\prime}}{p^\prime}}- a^{-\frac{\alpha}{\beta}}  e^{\alpha g(x)}\Big| \mathrm{d} x  \leq C(n,p,  \alpha,\beta,t) \cdot   \delta^{\sf HC}_{p ,t,\alpha,\beta}(g)^\frac{1}{2} ,
		\end{equation}
		where 
		\begin{equation}\label{theta_and_a_definition}
			\theta=\alpha\left(\frac{\beta-\alpha}{\beta t}\right)^{p^{\prime}-1}\ \  and\ \ a =  \int_{\R^n} e^{\beta \mathbf{Q}_{t} g(x)} \mathrm{d} x / \int_{\R^n} e^{-\theta \left(\frac{\beta}{\alpha}\right)^{p'}\cdot \frac{|x|^{p^\prime}}{p^\prime}} \mathrm{d} x. 
		\end{equation}
		Moreover, the exponent $\frac{1}{2}$ of $\delta^{\sf HC}_{p ,t,\alpha,\beta}(g)$ in \eqref{stab-HC} is sharp. 
	\end{theorem}
	
	The proof of relation \eqref{stab-HC}  relies on a slightly extended version of a  recent stability result for the Pr\'ekopa--Leindler inequality due to Figalli, van Hintum  and Tiba \cite[Corollary 1.7]{FigallivanHintumTiba_25}, stated as Theorem \ref{Theorem_Figalli_etal} below. The sharpness of the exponent $\frac{1}{2}$ of the deficit $\delta^{\sf HC}_{p ,t,\alpha,\beta}(g)$ on the right side of \eqref{stab-HC} is obtained  by using a suitable family of power-type test-functions, close to the extremizer from \eqref{HCI_extremizers}. 
	The exponent $\frac{1}{2}$ matches the same sharp exponent in the stability result of the  Pr\'ekopa--Leindler inequality as indicated in the one-dimensional example of \cite[Remark 1.8]{FigallivanHintumTiba_25}.
	In addition, Theorem \ref{main-theorem-HC} characterizes directly the equality in the hypercontractivity estimate  \eqref{HCI_ineq}, provided by the class of functions \eqref{HCI_extremizers},  see Corollary \ref{HC=}. 
	
	
	As we already pointed out, our next objective is to obtain a stability result for the  $L^p$-Euclidean logarithmic Sobolev inequality  for every $p>1$, which is intended to be derived  through the above hypercontractivity  stability. This step requires a limiting argument  in the hypercontractivity estimate as $t\to 0$, where $\beta=\beta(t)\to \alpha $ as $t\to 0$. Unfortunately, in its current form   the right-hand side of relation \eqref{stab-HC} does not allow such a limiting procedure as we do not know the exact behavior of the expression $C(n,p,\alpha, \beta, t)$ as $t\to 0$ and $\beta(t) \to \alpha$. The key ingredient that we can use to carry out this limiting procedure, will be a more refined estimate where a more precise information on the constant $C(n,p,\alpha, \beta, t)$ is available. More precisely, instead of \eqref{stab-HC} we would need  an expression of the form 
	$$\int_{\R^n}\left|e^{-\theta \cdot \frac{|x-x_0|^{p^\prime}}{p^\prime}}-a^{-\frac{\alpha}{\beta}} e^{\alpha  g\left( x\right)}\right| \mathrm{d} x  \leq C(n,p)\left(\frac{\delta^{\sf HC}_{p ,t,\alpha,\beta(t)}(g)}{1-\frac{\alpha}{\beta(t)}}\right)^{\frac{1}{\gamma}}$$
	for some $\gamma \geq 2$, instead of the bound  appearing in Theorem  \ref{main-theorem-HC}. The quest for the validity of a similar estimate with the exponent $\frac{1}{\gamma}= \frac{1}{2}$ in the context of  the Pr\'ekopa--Leindler stability has been formulated by  Figalli, van Hintum  and Tiba \cite[Remark 1.8]{FigallivanHintumTiba_25}. At the moment only  weaker  Pr\'ekopa--Leindler stability inequalities are available (see Theorems \ref{Theorem_Boroczky-De} and  \ref{Theorem_Figalli_Ramos} below) allowing us to establish the following result:  
	
	\begin{theorem} \label{Theorem2}
		Let $n \geq 1$ and  $p>1$. Then there exists a constant $ C(n,p) >0$ with the following property. Given any $t>0$, $0<\alpha<\beta$, and any concave function $g:\mathbb R^n\to \mathbb R$ such that $e^g\in L^\alpha(\mathbb R^n)$, and
		$
		\delta^{\sf HC}_{p ,t,\alpha,\beta}(g)  \ll 1,
		$  there exists $x_0 \in \mathbb{R}^n$ such that
		\begin{equation}\label{new-stability-Boroczky-De}
			\int_{\R^n}\left|e^{-\theta \cdot \frac{|x-x_0|^{p^\prime}}{p^\prime}}-a^{-\frac{\alpha}{\beta}} e^{\alpha  g\left( x\right)}\right| \mathrm{d} x  \leq C(n,p) \theta^{-\frac{n}{p'}} \left({\frac{\delta^{\sf HC}_{p ,t,\alpha,\beta}(g)}{\tau}}\right)^\frac{1}{19} ,
		\end{equation}
		where $\theta>0$ and
		$
		a>0
		$ 
		are from \eqref{theta_and_a_definition} and $\tau=\min \left(\frac{\alpha}{\beta}, 1-\frac{\alpha}{\beta}\right)$.
		
		In addition to the above assumptions, if $g:\mathbb R^n\to \mathbb R$ is radially symmetric, then 
		\begin{equation}\label{new-stability-Figalli-Ramos-HC}
			\int_{\R^n}\left|e^{-\theta \cdot \frac{|x|^{p^\prime}}{p^\prime}}-a^{-\frac{\alpha}{\beta}} e^{\alpha  g\left( x\right)}\right| \mathrm{d} x  \leq C(n,p) \theta^{-\frac{n}{p'}} \left({\frac{\delta^{\sf HC}_{p ,t,\alpha,\beta}(g)}{\tau}}\right)^\frac{1}{2} ,
		\end{equation}
		and the exponent  $\frac{1}{2}$ of the latter term in \eqref{new-stability-Figalli-Ramos-HC} is sharp.
	\end{theorem}

	The proofs of \eqref{new-stability-Boroczky-De} and \eqref{new-stability-Figalli-Ramos-HC} follow by two recent stability results for the Pr\'ekopa--Leindler inequality, the first by B\"or\"oczky and De \cite[Theorem 1.4]{BoroczkyDe}, the second by Figalli and Ramos \cite{FigalliRamos}, respectively. As expected, both estimates \eqref{new-stability-Boroczky-De} and \eqref{new-stability-Figalli-Ramos-HC} come  at a certain cost in comparison with Theorem \ref{main-theorem-HC}. Indeed, in \eqref{new-stability-Boroczky-De}, the function $g$ 
	is required to be concave and the sharpness of the exponent is lost, dropping from order $1/2$ to $1/19,$ while in \eqref{new-stability-Figalli-Ramos-HC}, even though we have the sharpness of the exponent $1/2$,  the price is paid by the concavity and radial symmetry of the function $g$. 
	
	
	The main advantage of Theorem \ref{Theorem2} is that when $\beta:=\beta(t)$ is suitably chosen,  there exists $c>0$ such that  $\tau:=\tau(t)\sim ct$
	for $0<t\ll 1$, thus we can take a meaningful limit in \eqref{new-stability-Boroczky-De} and \eqref{new-stability-Figalli-Ramos-HC} as $t\to 0$,  which will
	imply  stability results for the $L^p$-Euclidean logarithmic Sobolev inequality.   
	
	Let us recall that for a given $n \geq 1$ and $p>1$, the $L^p$-Euclidean logarithmic Sobolev inequality
	in $\R^n$ can be stated as
	\begin{equation} \label{LSI_ineq}
		\frac{1}{\|f\|_p^p} \Ent(|f|^p) 
		\leq \frac{n}{p} \log \left( \mathcal{L}_{n,p} \frac{1}{\|f\|_p^p} \int_{\R^n} |\nabla f|^p \, \mathrm{d} x \right), \ \ \forall f\in W^{1,p}(\mathbb R^n),
	\end{equation}
	where
	$$\Ent(|f|^p) = \int_{\R^n} |f|^p \log |f|^p \, \mathrm{d} x - \int_{\R^n} |f|^p \mathrm{d} x \cdot\log \int_{\R^n} |f|^p \mathrm{d} x$$
	stands for the entropy of $|f|^p$, 
	and the constant
	\begin{equation} \label{LSI_constant}
		\mathcal{L}_{n,p}=\frac{p}{n}\left(\frac{p-1}{e}\right)^{p-1}\left(\Gamma\left(\frac{n}{p^{\prime}}+1\right) \omega_n\right)^{-\frac{p}{n}}
	\end{equation}
	is optimal. Hereafter,  $p'=\frac{p}{p-1}$ is the conjugate of $p$. 
	Furthermore, equality holds in \eqref{LSI_ineq} if and only if
		\begin{equation}\label{Gaussian-extremizer-01}
			f(x)=c e^{-\frac{1}{C}|x-x_0|^{p^\prime}}, \ x\in \mathbb R^n,
		\end{equation}
		%
		%
	for some $x_0 \in \R^n$, $c\in \mathbb R\setminus \{0\}$ and $C >0$, see del Pino and Dolbeault \cite{dPD} for $1<p<n$, and Balogh, Don and Krist\'aly \cite{BDK} for the general case $p>1$ and $n\geq 1$.

	%
	%
	%
	To state the stability of the $L^p$-Euclidean logarithmic Sobolev inequality, let us introduce for every $f\in W^{1,p}(\mathbb R^n)\setminus \{0\}$
	the  $L^p$-\textit{logarithmic Sobolev deficit} associated to  inequality \eqref{LSI_ineq}, i.e., 	
	\begin{equation} \label{LSI_deficit}
		\delta_p^{\sf LSI}(f) = \frac{n}{p} \log \left(\mathcal{L}_{n,p} \frac{1}{\|f\|_p^p}\int_{\mathbb{R}^n}|\nabla f|^p \mathrm{d} x \right) - \frac{\Ent\left(|f|^p\right)}{\|f\|_p^p} \geq 0,
	\end{equation}
	where $\mathcal{L}_{n,p}$ is from \eqref{LSI_constant}. 
	
	
	\begin{theorem}\label{Theorem-LSI-main}
		Let $n \geq 1$ and $p>1$. There exists a constant $C(n) >0$ with the following property. For every  log-concave function $f \in W^{1,p}(\mathbb R^n)
		$ satisfying the growth 
		condition 
		\begin{equation}\label{one-sided}
			f(x)\geq c_1e^{-c_2|x|^{p'}},\ x\in \mathbb R^n,
		\end{equation}
		for some $c_1,c_2>0$, there exists  $x_0 \in \mathbb{R}^n$ such that    
		\begin{equation}\label{lsi-estimate}
			\int_{\R^n}\left| {\sf C}_2  e^{-\frac{p}{{\sf C}_1} |x-x_0|^{p^\prime}} - \frac{f^p(x)}{\|f\|_p^p} \right| \mathrm{d} x  \leq C(n)
			\delta_p^{\sf LSI}(f)^\frac{1}{19}, 
		\end{equation}
		where
		\begin{equation}\label{C-1-C-2}
			{\sf C}_1: = p' \left( \frac{n}{p} \right)^{p^\prime-1} {\|f\|_p^{p^\prime}} \left( \int_{\mathbb{R}^n} |\nabla f|^p \,        \mathrm{d}x \right)^{1-p^\prime}\ \ 
			and \ \ 
			{\sf C}_2=\left(\left(\frac{{\sf C}_1}{p}\right)^{\frac{n}{p^\prime}} \Gamma\left(\frac{n}{p^\prime}+1\right) \omega_n\right)^{-1}.
		\end{equation}
		
		In addition, we also have that
		\begin{equation}\label{lsi-estimate-Schwarz}
			\int_{\R^n}\left| {\sf C}_2^\star  e^{-\frac{p}{{\sf C}_1^\star} |x|^{p^\prime}} - \frac{{(f^\star)}^p(x)}{\|f\|_p^p} \right| \mathrm{d} x  \leq C(n)
			\delta_p^{\sf LSI}(f)^\frac{1}{2}, 
		\end{equation}
		and the exponent  $\frac{1}{2}$  in \eqref{lsi-estimate-Schwarz} is sharp, 	where $f^\star$ stands for the Schwarz-rearrangement of $f$ and ${\sf C}_1^\star>0$ and  ${\sf C}_2^\star>0$ are the corresponding values from \eqref{C-1-C-2}, computed for $f^\star$ instead of $f$.
	\end{theorem}
	
	The proof of Theorem \ref{Theorem-LSI-main} is based on Theorem \ref{Theorem2}. Namely,  \eqref{lsi-estimate} follows by \eqref{new-stability-Boroczky-De}, together with a careful limiting argument that establishes a key relationship between the hypercontractivity and $L^p$-logarithmic Sobolev deficits, respectively, see Proposition \ref{prop_1}; this argument requires also the  growth assumption \eqref{one-sided}.  A similar argument works also for \eqref{lsi-estimate-Schwarz} by means of \eqref{new-stability-Figalli-Ramos-HC}, combined with the P\'olya--Szeg\H o inequality and the fact that log-concavity is preserved under the Schwarz-rearrangement. Clearly, once $f$ is radially symmetric and log-concave, relation \eqref{lsi-estimate-Schwarz} is valid without Schwarz-rearrangement (thus, we can simply write $f$ instead of $f^\star$), see also \eqref{lsi-estimate-radialisra}. The sharpness of the exponent  ${1}/{2}$  in \eqref{lsi-estimate-Schwarz} will be shown by a special class of Gaussian-type functions, close to the function \eqref{Gaussian-extremizer-01}, by using basic properties of the Gamma, Digamma and Trigamma functions.   
	
	Theorem \ref{Theorem-LSI-main} implies  the equality case in the $L^p$-Euclidean logarithmic Sobolev inequality for every $p>1$ whenever the extremal is assumed to be log-concave; the full characterization is given by Balogh, Don and Krist\'aly \cite{BDK} for the general case $p>1$ and $n\geq 1$.

	A natural question arises concerning the limit situation in Theorem \ref{Theorem-LSI-main} whenever $p\searrow 1$. In this case, $\mathcal{L}_{n,p}\to n^{-1}\omega_n^{-1/n}$, while   \eqref{LSI_ineq} reduces to the  $L^1$-\textit{Euclidean logarithmic Sobolev inequality} which holds for the larger class of functions with bounded variation ${\sf BV}(\mathbb R^n)$. In this case 
 the family of extremizers is provided by the characteristic functions of balls in $\mathbb R^n$ (of any radius and center), see  Beckner \cite{Beckner} and  Ledoux \cite{Ledoux}. As expected, the formal  limiting procedures in \eqref{lsi-estimate}   and \eqref{lsi-estimate-Schwarz} when $p\searrow 1$ provide  the expected characteristic functions of balls, but the class of admissible functions  after the limiting argument in  \eqref{one-sided} requires further analysis; for details, see Remark \ref{remark-Beckner}.                    

It is well known that for $p=2$, the $L^2$-Euclidean logarithmic Sobolev inequality \eqref{LSI_ineq} is equivalent to the \textit{Gaussian logarithmic Sobolev inequality}. 
Therefore, in view of Theorem \ref{Theorem-LSI-main}, it is natural to ask for the stability of the Gaussian logarithmic Sobolev inequality.	To answer this question, we can follow the same scheme \hyperref[S]{\textbf{(S)}}  in the Gaussian setting as well as in the Euclidean framework; see Section \ref{section-5}. More precisely, by stability of the Pr\'ekopa--Leindler inequality, we state stability results for the Gaussian hypercontractivity estimate (see Theorems \ref{theorem-G-HC-1} and \ref{G-Theorem2}), as instances of the first implication in  \hyperref[S]{\textbf{(S)}}. These will yield in turn -- as the second implication in  \hyperref[S]{\textbf{(S)}} --  a stability of the Gaussian logarithmic Sobolev inequality (see Theorem \ref{Gaussian-Theorem-LSI-main}).


		The structure of the paper is as follows. In Section \ref{section-2} we recall  the stability results for the Pr\'ekopa--Leindler inequality in the specific form that we need them in the proofs  of Theorems \ref{main-theorem-HC} and \ref{Theorem2}.
		Section \ref{section-30} is devoted to the proof of 
		Theorems \ref{main-theorem-HC} and \ref{Theorem2}, respectively. In Section \ref{section-4} we prove  Theorem \ref{Theorem-LSI-main},  which is based on Proposition \ref{prop_1},  stating a connection 
		between the hypercontractivity deficit and $L^p$-logarithmic Sobolev deficit, respectively. In Section \ref{section-5} we prove stability results for the Gaussian hypercontractivity estimates and Gaussian logarithmic Sobolev inequality. 	 Section \ref{section-6} is devoted to final comments and open questions.
		In the Appendix we derive a Hamilton--Jacobi equation at the origin for the Hopf--Lax semigroup as well as we establish a  crucial property of the Trigamma function that is used to prove the sharpness of the exponent $1/2$ in the $L^p$-Euclidean logarithmic Sobolev inequality \eqref{lsi-estimate-Schwarz}.
		
		\medskip 
		
		\noindent {\bf Acknowledgements:} We would like to thank K\'aroly B\"or\"oczky, Alessio Figalli and Jo$\tilde{\rm a}$o Pedro Ramos for discussions related to the subject of this paper.

		\section{Preliminaries: Stability in the Pr\'ekopa--Leindler inequality}\label{section-2} 
		
		
		Let $u, v, w: \mathbb{R}^n \rightarrow \mathbb{R}_{+}$ be integrable functions and $\lambda\in (0,1)$ such that
		\begin{equation}\label{PL-feltetel}
			w(\lambda x+(1-\lambda) y) \geq u(x)^\lambda v(y)^{1-\lambda} \ \ {\rm for\ every}\ \ x, y \in \mathbb{R}^n. 
		\end{equation}
		The Pr\'ekopa--Leindler inequality asserts that 
		$$\ds\int_{\mathbb{R}^n} w \geq  \left(\int_{\mathbb{R}^n} u \right)^\lambda \left(\int_{\mathbb{R}^n} v \right)^{1-\lambda} .$$
		Moreover, equality holds in the latter inequality if and only if $u,v,w$  are log-concave functions up to a set of measure zero and there exists $x_0\in \mathbb R^n$ such that
		\begin{itemize}
			\item $u(x) = a v(x- x_0)$  for a.e.  $x \in \mathbb{R}^n$, and
			\item $w(x) = a^{\lambda} v(x-\lambda x_0)$  for a.e.  $x \in \mathbb{R}^n$, 
		\end{itemize}
		where $a = \int_{\R^n} u / \int_{\R^n} v$, see Dubuc \cite{Dubuc} and also  Balogh and Krist\'aly \cite{BK-Advances}. 
		
		In the sequel, we recall some recent stability results for the Pr\'ekopa--Leindler inequality which are the key tools in the proof of our main results. To do this, we assume that for some $\varepsilon \geq 0$,
		\begin{equation}\label{PL-epsilon}
			\ds\int_{\mathbb{R}^n} w =(1+\varepsilon) \left(\int_{\mathbb{R}^n} u \right)^\lambda \left(\int_{\mathbb{R}^n} v \right)^{1-\lambda}.
		\end{equation}
		

		According to Figalli, van Hintum  and Tiba \cite[Corollary 1.7]{FigallivanHintumTiba_25}, we have the following general stability result: 
		
		\begin{theorem}[Figalli, van Hintum and Tiba \cite{FigallivanHintumTiba_25}]\label{Theorem_Figalli_etal}
			Let $n \in \mathbb{N}$ and $\lambda \in(0,1 / 2]$. Then there exists an absolute constant $C = C(n, \lambda) > 0$ with the following property. Let $u, v, w: \mathbb{R}^n \rightarrow \mathbb{R}_{+}$ be integrable functions such that \eqref{PL-feltetel} and \eqref{PL-epsilon} hold for some $\varepsilon\geq 0.$ Then 
			there exist translation vectors $x_0,y_0 \in \R^n$  such that
			\begin{equation}\label{u-v-w}
				\int_{\mathbb{R}^n}|u(x) - a  v(x-x_0)| \mathrm{d} x +	a\int_{\mathbb{R}^n}|a^{-\lambda}w(x) -   v(x-y_0)| \mathrm{d} x \leq C \sqrt{\varepsilon} \int_{\mathbb{R}^n} u \mathrm{d} x,
			\end{equation} 
			where $a = \int_{\R^n} u / \int_{\R^n} v$. 
		\end{theorem} 
		\begin{proof}
			Let us introduce the functions $f, g, h: \mathbb{R}^n \rightarrow \mathbb{R}_{+}$ defined by 
			$f = a^{\lambda-1} u, \quad g = a^{\lambda} v, \quad h = w. $     
			Clearly, we have that 
			\begin{itemize}
				\item $\ds\int_{\mathbb{R}^n} f=\int_{\mathbb{R}^n} g;$
				\item $h(\lambda x+(1-\lambda) y) \geq f(x)^\lambda g(y)^{1-\lambda}$ for all $x, y \in \mathbb{R}^n$, and
				\item $\ds\int_{\mathbb{R}^n} h =(1+\varepsilon) \int_{\mathbb{R}^n} f $.
			\end{itemize}
			According to Figalli, van Hintum  and Tiba \cite[Corollary 1.7]{FigallivanHintumTiba_25}, there exists a log-concave function $l:\mathbb{R}^n \to \mathbb{R}_+$ and a constant $C = C(n, \lambda)>0$ such that, up to compositions of $f$ and $g$ with translations we have the estimate: 
			\begin{equation}\label{f-g}
				\ds \int_{\mathbb{R}^n} (|f-l|+|g-l|)\mathrm{d} x\leq C\sqrt{\varepsilon}\int_{\mathbb{R}^n}f\mathrm{d} x.
			\end{equation}
			In addition, as in the proof of Theorem 1.6 of \cite{FigallivanHintumTiba_25}, if $k:=\min\{f,g\}$, one has that\footnote{Here we use the notation $O_{n,\lambda}(\sqrt{\varepsilon})$ to indicate that the expression in question is bounded from above (respectively, from below)
				by the quantity 
				$C(n, \lambda)\sqrt{\varepsilon}$, where $C(n, \lambda)$ is a positive constant that  may depend on $n$ and $\lambda$.}
			\begin{equation}\label{h-k-l}
				h\geq k,\ \ \ds \int_{\mathbb{R}^n} k\mathrm{d} x\geq  (1-O_{n,\lambda}(\sqrt{\varepsilon}))\int_{\mathbb{R}^n} f\mathrm{d} x\ \ {\rm and}\ \  \ds \int_{\mathbb{R}^n} |k-l|\mathrm{d} x\leq O_{n,\lambda}(\sqrt{\varepsilon})\int_{\mathbb{R}^n} k\mathrm{d} x.
			\end{equation}
			Therefore, 
			\begin{eqnarray*}
				\ds \int_{\mathbb{R}^n} |h-k|\mathrm{d} x&=&\ds \int_{\mathbb{R}^n} (h-k)\mathrm{d} x\\&\leq& (1+\varepsilon) \int_{\mathbb{R}^n} f \mathrm{d} x -(1-O_{n,\lambda}(\sqrt{\varepsilon}))\int_{\mathbb{R}^n} f\mathrm{d} x=O_{n,\lambda}(\sqrt{\varepsilon})\int_{\mathbb{R}^n} f\mathrm{d} x,
			\end{eqnarray*}
			which implies -- combined with \eqref{h-k-l} -- that
			$$\ds \int_{\mathbb{R}^n} |h-l|\mathrm{d} x\leq \ds \int_{\mathbb{R}^n} (|h-k|+|k-l|)\mathrm{d} x\leq O_{n,\lambda}(\sqrt{\varepsilon})\int_{\mathbb{R}^n} f\mathrm{d} x.$$
			Now, the latter relation together with \eqref{f-g} implies the existence of a constant $C = C(n, \lambda)>0$ such that, up to translations, 
			$$	\ds \int_{\mathbb{R}^n} (|f-l|+|g-l|+|h-l|)\mathrm{d} x\leq C\sqrt{\varepsilon}\int_{\mathbb{R}^n}f\mathrm{d} x,$$
			which, in particular, gives that 
			$$	\ds \int_{\mathbb{R}^n} (|f-g|+|h-g|)\mathrm{d} x\leq C\sqrt{\varepsilon}\int_{\mathbb{R}^n}f\mathrm{d} x.$$
			This estimate is exactly  \eqref{u-v-w}, which concludes the proof. 
		\end{proof}
		
		We now recall the main result from B\"or\"oczky and De  \cite[Theorem 1.4]{BoroczkyDe} which gives a  more precise estimate as \eqref{u-v-w}, but requiring the \textit{log-concavity} of the function $w$: 
		
		\begin{theorem}[B\"or\"oczky and De  \cite{BoroczkyDe}] \label{Theorem_Boroczky-De}
			For every $n \in \mathbb{N}$, there exists a constant $c = c(n) >0$ with the following property. Let $\tau \in(0,1 / 2]$, $\tau\leq \lambda\leq 1-\tau$, and $u, v, w: \mathbb{R}^n \rightarrow \mathbb{R}_{+}$ be integrable functions such that $w$ is log-concave, and \eqref{PL-feltetel} and \eqref{PL-epsilon} hold for some $\varepsilon\in (0,1].$
			Then, there exists a translation vector $x_0 \in \R^n$  such that
			$$
			\int_{\mathbb{R}^n}|v(x-x_0) - a^{-\lambda}w(x) | \mathrm{d} x \leq c(n) \left(\frac{\varepsilon}{\tau}\right)^{1/19} \int_{\mathbb{R}^n} v(x) \mathrm{d} x,
			$$
			where $a = \int_{\R^n} u / \int_{\R^n} v$.
		\end{theorem} 
		
		\begin{remark}\rm\label{remark-constant} The constant $c(n)$ is $c^n n^n$, where $c>0$ is a universal constant,  	see  \cite{BoroczkyDe}. 
		\end{remark}
		
		In this context, another relevant result is due to Figalli and Ramos \cite[Theorem 3]{FigalliRamos}, which states a sharp stability estimate for \textit{radially symmetric} and \textit{log-concave} functions. More precisely, we have: 
		
		\begin{theorem}[Figalli and Ramos \cite{FigalliRamos}] \label{Theorem_Figalli_Ramos}
			For every $n \in \mathbb{N}$, there exists a dimensional  constant $ C(n) >0$ with the following property.
			Let $0<\lambda<1$, and $u, v, w: \mathbb{R}^n \rightarrow \mathbb{R}_{+}$ be radially symmetric functions such that  either $w$ is log-concave, or both $u$ and $v$ are log-concave. Furthermore, suppose that \eqref{PL-feltetel} and \eqref{PL-epsilon} hold for some $\varepsilon\in (0,1].$
			Then there exist a radially symmetric log-concave function $h: \mathbb{R}^n \rightarrow \mathbb{R}_{+}$ such that
			$$
			\begin{aligned}
				\int_{\R^n}| u(x) - a^{1-\lambda}  h\left(x\right)| \mathrm{d} x & \leq C(n)\left(\frac{\varepsilon}{\tau}\right)^{1 / 2} \int_{\mathbb{R}} u(x) \mathrm{d} x, \\
				\int_{\R^n}|v(x)- a^{-\lambda}h\left(x\right)| \mathrm{d} x & \leq C(n)\left(\frac{\varepsilon}{\tau}\right)^{1 / 2} \int_{\mathbb{R}} v(x)  \mathrm{d} x, \\
				\int_{\mathbb{R}}|w(x)-h(x)|  \mathrm{d} x & \leq C(n)\left(\frac{\varepsilon}{\tau}\right)^{1 / 2} \int_{\R^n} w(x)  \mathrm{d} x,
			\end{aligned}
			$$
			where $a = \int_{\R^n} u / \int_{\R^n} v$ and $\tau= \min(\lambda, 1-\lambda)$. 
		\end{theorem}
		
		In addition to the above results, throughout the whole paper, we shall use the following integral formula 
		\begin{equation}\label{integral-formula}
			\int_{\R^n} e^{-M|x|^{q}} \mathrm{d} x =\Gamma\left(\frac{n}{q} +1\right) \omega_n \cdot M^{-\frac{n}{q}},
		\end{equation}
		for every $q >1$ and $M>0$.

		\section{Proof of Theorems \ref{main-theorem-HC} and \ref{Theorem2}}\label{section-30}
		
		In this section we prove stability of the hypercontractivity estimates for the Hopf--Lax semigroup $(\mathbf{Q}_{t})_{t \geq 0}$, i.e., Theorems \ref{main-theorem-HC} and \ref{Theorem2}. 
		Recall that for a function $g:\R^n\to \mathbb R$,   $(\mathbf{Q}_t)_{t \geq 0}$ is the family of  operators given by 
		\begin{equation}\label{Q-t-definition}
			\left\{
			\begin{array}{l}
				\mathbf{Q}_t g(x)=\inf _{y \in \R^n}\left\{g(y)+\frac{|x-y|^{p^{\prime}}}{p^{\prime} t^{p^{\prime}-1}}\right\}, \quad t>0, \quad x  \in \mathbb{R}^n, \\
				\mathbf{Q}_0 g(x)=g(x), \quad x  \in \mathbb{R}^n.
			\end{array}
			\right.
		\end{equation} 
		It can be proven that under a certain regularity assumptions on $g$,  the function $u(x, t)=\mathbf{Q}_{t} g(x)$ is the viscosity solution to the   Hamilton--Jacobi initial value problem 
		$$
		\left\{
		\begin{array}{l}
			\frac{\partial u}{\partial t}(x, t) + \frac{|\nabla u(x, t)|^p}{p} =0, \quad t>0, \quad x  \in \mathbb{R}^n ,\\
			u(x, 0)=g(x), \quad x  \in \mathbb{R}^n,
		\end{array}\right.
		$$
		see Evans \cite{Evans}.

		\subsection{Proof of Theorem \ref{main-theorem-HC}.}  
		We divide the proof into two steps. 
		
		\textbf{Step 1:} \textit{proof of \eqref{stab-HC}.} 
		Let $0<\alpha<\beta$ and $\lambda \in (0,1)$. By the definition of the Hopf--Lax semigroup $(\mathbf{Q}_{t})_{t> 0}$, for all $t >0$ and all $x,y \in \R^n$, we have that
		\begin{equation}\label{Q-t-estimate}
			\mathbf{Q}_t g(x) \leq g\left(x+\left(\frac{1-\lambda}{\lambda}\right) y\right)+ \left(\frac{1-\lambda}{\lambda}\right)^{p^{\prime}} \frac{|y|^{p^{\prime}}}{p^{\prime} t^{p^{\prime}-1}}.
		\end{equation}
		Multiplying the above relation by $\alpha$ and taking exponentials of both sides, we obtain that 
		\begin{equation*}
			e^{\alpha \mathbf{Q}_t g(x)} \cdot e^{-\alpha \left(\frac{1-\lambda}{\lambda}\right)^{p^{\prime}}  \frac{|y|^{p^{\prime}}}{p^{\prime} t^{p^{\prime}-1}}} 
			\leq e^{\alpha g\left(x+\left(\frac{1-\lambda}{\lambda}\right) y\right) } 
			=  e^{\alpha g\left(\frac{1}{\lambda}\left(\lambda x+ (1-\lambda)y\right)\right)} .
		\end{equation*}
		
		We first assume that  $\alpha\leq \frac{\beta}{2}$; thus, choosing $\lambda \coloneqq \frac{\alpha}{\beta} \in (0,\frac{1}{2}]$ yields
		\begin{equation*}
			\left(e^{\beta \mathbf{Q}_t g(x)}\right)^\lambda \cdot \left(e^{- \beta \left(\frac{\beta-\alpha}{\alpha}\right)^{p^{\prime}-1}  \frac{|y|^{p^{\prime}}}{p^{\prime} t^{p^{\prime}-1}}}\right)^{1-\lambda} 
			\leq e^{\alpha g\left(\frac{\beta}{\alpha}\left(\lambda x+ (1-\lambda)y\right)\right)} .
		\end{equation*}
		Therefore, considering the functions  $u,v,w: \R^n \to \R_+$, 
		\begin{equation}\label{definition-u-v-w}
			u(x)=e^{\beta \mathbf{Q}_{t} g(x)}, \quad \quad
			v(x)=e^{-\theta_0 \cdot \frac{|x|^{p^\prime}}{p^{\prime}}}, \quad \quad
			w(x)=e^{\alpha g\left(\frac{\beta}{\alpha} x\right)},
		\end{equation}
		where
		$$
		\theta_0=\beta\left(\frac{\beta-\alpha}{\alpha t}\right)^{p^{\prime}-1} ,
		$$
		it follows that 
		\begin{equation} \label{condition_PL_1}
			u(x)^{\lambda}  v(y)^{1-\lambda} 
			\leq  w( \lambda x+(1-\lambda) y) .
		\end{equation}
		In addition, all functions $u,v$ and $w$ are integrable by assumption.

		
		Next, let $\varepsilon \coloneqq \delta^{\sf HC}_{p ,t,\alpha,\beta}(g) $. Then, by \eqref{HCI_deficit}, it follows that 
		\begin{equation}  \label{eq_3}
			\left(\mathcal{C}_{p, t, \alpha, \beta} \right)^\alpha
			\left\|e^g\right\|^\alpha_{\alpha} = (1+\varepsilon)\left\|e^{\mathbf{Q}_{t} g}\right\|_{\beta}^\alpha  ,
		\end{equation}
		which, in turn, is equivalent to the relation
		\begin{equation}  \label{eq_4}
			\int_{\R^n} w = (1+\varepsilon) \left(\int_{\R^n} u\right)^{\lambda}\left(\int_{\R^n} v\right)^{1-\lambda} ,
		\end{equation}
		with $\lambda=\frac{\alpha}{\beta}$.
		Indeed,
		by using the integral formula \eqref{integral-formula},
		we obtain that
		\begin{equation} \label{eq_v}
			\int_{\R^n} v = \int_{\R^n} e^{-\frac{\theta_0}{p^{\prime}}|x|^{p^\prime}} \mathrm{d}x = \Gamma\left(\frac{n}{p^{\prime}} +1\right)  \omega_{n} \left(\frac{\theta_0}{p^{\prime}}\right)^{-\frac{n}{p^\prime}} .
		\end{equation}
		On the other hand, by a change of variables, we have that
		$$
		\int_{\R^n} w =\int_{\R^n} e^{\alpha g\left(\frac{\beta}{\alpha} x\right)} \mathrm{d} x = \left(\frac{\alpha}{\beta}\right)^n\left\|e^g\right\|_{\alpha}^\alpha .
		$$
		Accordingly, a straightforward calculation confirms the equivalence between \eqref{eq_3} and \eqref{eq_4}.
		
		Now, we are in the position to  apply Theorem \ref{Theorem_Figalli_etal} to the functions $u,v$ and $w$, obtaining in particular that there exist $x_0\in \mathbb R^n$ and  a constant $C_0=C_0(n,
		\lambda)>0$ such that
		\begin{align*}
			\int_{\mathbb{R}^n}|a^{-\lambda}w(x) -   v(x-x_0)| \mathrm{d} x \leq C_0 \sqrt{\varepsilon} \int_{\mathbb{R}^n} v \mathrm{d} x,
		\end{align*} 
		where 
		$$a =  \frac{\int_{\R^n} u}{\int_{\R^n} v} = \frac{\int_{\R^n} e^{\beta \mathbf{Q}_{t} g(x)} \mathrm{d} x}{\int_{\R^n} e^{-\theta_0 \cdot \frac{|x|^{p^\prime}}{p^\prime}} \mathrm{d} x}.
		$$
		The latter inequality is equivalent to 
		\begin{equation}\label{stab-HC-11}
			\int_{\R^n}\Big|e^{-\theta_0 \cdot \frac{|x-x_0|^{p^\prime}}{p^\prime}}- a^{-\frac{\alpha}{\beta}}  e^{\alpha g(\frac{\beta}{\alpha}x)}\Big| \mathrm{d} x  \leq C_0  \, \delta^{\sf HC}_{p ,t,\alpha,\beta}(g)^\frac{1}{2}  \int_{\R^n} e^{-\frac{\theta_0}{p^{\prime}}|x|^{p^\prime}} \mathrm{d}x.
		\end{equation}
		According to	\eqref{eq_v} and a change of variables in the left hand side of \eqref{stab-HC-11} (note also that $\theta_0=\left({\beta}/{\alpha}\right)^{p'}\theta $), we obtain the estimate \eqref{stab-HC} by choosing 
		\begin{equation}\label{C-constant}
			C(n,p,  \alpha,\beta,t):= C_0\Gamma\left(\frac{n}{p^{\prime}} +1\right)  \omega_{n} \cdot\left(p^{\prime}\right)^{\frac{n}{p^\prime}} \alpha^{-\frac{n}{p^\prime}}  \left(\frac{\beta-\alpha}{\beta t} \right)^{-\frac{n}{p}}, 
		\end{equation}
		where $C_0=C_0(n,
		\alpha/\beta)>0$ is from  Theorem \ref{Theorem_Figalli_etal}. 
		
		In the complementary case when $\alpha> \frac{\beta}{2}$, we can obtain  inequality \eqref{condition_PL_1} by starting from relation 
		$$
		\mathbf{Q}_t g(x) \leq g\left(x+\left(\frac{\lambda}{1-\lambda}\right) y\right)+ \left(\frac{\lambda}{1-\lambda}\right)^{p^{\prime}} \frac{|y|^{p^{\prime}}}{p^{\prime} t^{p^{\prime}-1}},
		$$ instead of \eqref{Q-t-estimate},  
		and choosing $\lambda \coloneqq 1-\frac{\alpha}{\beta} \in (0,\frac{1}{2})$. The rest of the proof is similar to the one above. We leave the details to the interested reader.

		\textbf{Step 2:} \textit{sharpness of the exponent $\frac{1}{2}$ in \eqref{stab-HC}.} Let us choose $\alpha=1=t$ and $\beta=p>1$. With these choices, one has that
		$\theta=(p')^{1-p'}$.
		
		For every $0<\varepsilon<\frac{1}{p'}-\frac{1}{(p')^{p'}}$, we introduce the family of functions  
		\begin{equation}\label{g-epsilon}
			g_\varepsilon(x)=-\left(\frac{1}{(p')^{p'}}+\varepsilon\right)|x|^{p'},\ \ x\in \mathbb R^n.
		\end{equation}
		We claim that 
		\begin{equation}\label{1-claim}
			\delta^{\sf HC}_{p ,1,1,p}(g_\varepsilon)=\frac{n}{2}p^{\frac{p+1}{p-1}}(p-1)^\frac{p-3}{p-1}\varepsilon^2 +o(\varepsilon^2).
		\end{equation}
		For convenience, let 
		\begin{equation}\label{b-deinfition}
			b_\varepsilon:=p'\left(\frac{1}{(p')^{p'}}+\varepsilon\right).
		\end{equation}
		By the range of $\varepsilon>0$, it follows that $b_\varepsilon<1$. First, an elementary computation shows that 
		\begin{equation}\label{g-eps}
			\|e^{g_\varepsilon}\|_1=\Gamma\left(\frac{n}{p^{\prime}} +1\right) \omega_n \left(\frac{b_\varepsilon}{p'}\right)^{-\frac{n}{p^\prime}}.
		\end{equation}
		Moreover, a direct calculation gives that 
		$$\mathbf{Q}_1g_\varepsilon(x)=-\frac{b_\varepsilon}{p'(1-b_\varepsilon^{p-1})^{p'-1}}|x|^{p'},\ \ x\in \mathbb R^n,$$
		see also Balogh, Don and Krist\'aly \cite[relation (5.16)]{BDK}, which implies that
		\begin{equation}\label{Q1-g-eps}
			\|e^{\mathbf{Q}_1g_\varepsilon}\|_p^p=\Gamma\left(\frac{n}{p^{\prime}} +1\right) \omega_n \left(\frac{pb_\varepsilon}{p'(1-b_\varepsilon^{p-1})^{p'-1}}\right)^{-\frac{n}{p^\prime}}.
		\end{equation}
		Using  expressions \eqref{g-eps}  and \eqref{Q1-g-eps}, we obtain 
		$$\delta^{\sf HC}_{p ,1,1,p}(g_\varepsilon)=\mathcal{C}_{p, 1, 1, p} \frac{
			\left\|e^{g_\varepsilon}\right\|_1 }{\left\|e^{\mathbf{Q}_1 g_\varepsilon}\right\|_{p}} -1=p^{-\frac{n}{p}}(p-1)^\frac{n}{p'p}b_\varepsilon^{-\frac{n}{(p')^2}}(1-b_\varepsilon^{p-1})^{-\frac{n}{p^2}}-1.$$ Relation \eqref{b-deinfition} and a straightforward asymptotic analysis directly imply the validity of \eqref{1-claim}. 
		
		In the sequel, we focus on the left hand side of \eqref{stab-HC}; in fact, according to \eqref{1-claim}, 
		we complete the proof once we are able to prove the existence of a universal constant $C>0$ such that  for every $x_0\in \mathbb R^n$ and $\varepsilon>0$ small enough, 
		\begin{equation}\label{epsilon-sharp}
			\int_{\R^n}\Big|e^{-\theta \cdot \frac{|x-x_0|^{p^\prime}}{p^\prime}}- a^{-\frac{1}{p}}  e^{g_\varepsilon(x)}\Big| \mathrm{d} x  \geq C \varepsilon,
		\end{equation}
		where $$a= \frac{\int_{\R^n} e^{p \mathbf{Q}_{1} g_\varepsilon(x)} \mathrm{d} x}{\int_{\R^n} e^{-\theta_0 \cdot \frac{|x|^{p^\prime}}{p^\prime}} \mathrm{d} x}=\left(\frac{b_\varepsilon}{(p-1)^{p'-1}(1-b_\varepsilon^{p-1})^{p'-1}}\right)^{-\frac{n}{p'}}.$$

		If we denote by
		$$z:=z_\varepsilon=p'b_\varepsilon^{p-1},$$
		then $z_\varepsilon>1$, $z_\varepsilon\to 1$ as $\varepsilon \to 0$, and \eqref{epsilon-sharp} can be equivalently rewritten into the form 
		\begin{equation}\label{z-sharp}
			\int_{\R^n}\Big|e^{-(p')^{-p'} |x-x_0|^{p^\prime}}- \left(\frac{z}{p(1-\frac{z}{p'})}\right)^\frac{n}{p^2}  e^{-(p')^{-p'}z^\frac{1}{p-1} |x|^{p'}}\Big| \mathrm{d} x  \geq \tilde C (z-1),
		\end{equation}
		for every $z>1$ close to 1 (where $\tilde C>0$ is also universal constant). We distinguish two cases, depending on the positions of the peaks of the Gaussian functions appearing in \eqref{z-sharp}.  
		
		\textit{Case 1:} $x_0=0$. In this case we can integrate in spherical coordinates such that relation  \eqref{z-sharp} reduces to $$\int_0^\infty\Big|e^{-(p')^{-p'} r^{p^\prime}}- \left(\frac{z}{p(1-\frac{z}{p'})}\right)^\frac{n}{p^2}   e^{-(p')^{-p'}z^\frac{1}{p-1} r^{p'}}\Big|r^{n-1} \mathrm{d} r  \geq (n\omega_n)^{-1}\tilde C (z-1)$$ for every $z>1$ close to 1. 
		Let $$I(z,r)= \left(\frac{z}{p(1-\frac{z}{p'})}\right)^\frac{n}{p^2}    e^{-(p')^{-p'}z^\frac{1}{p-1} r^{p'}} $$ for every  $r>0$ and $z\geq  1$ close to 1. 
		By Fatou's lemma we have 
		\begin{eqnarray*}
			J&:=&\liminf_{z\to 1}\frac{\ds\int_0^\infty\Big|e^{-(p')^{-p'} r^{p^\prime}}- \left(\frac{z}{p(1-\frac{z}{p'})}\right)^\frac{n}{p^2}   e^{-(p')^{-p'}z^\frac{1}{p-1} r^{p'}}\Big|r^{n-1} \mathrm{d} r}{z-1}
			\\&=&\liminf_{z\to 1}{\ds\int_0^\infty\frac{|I(z,r)-I(1,r)|}{z-1}r^{n-1} \mathrm{d} r}
			\\&\geq& \ds\int_0^{\infty}\liminf_{z\to 1}\frac{|I(z,r)-I(1,r)|}{z-1}r^{n-1} \mathrm{d} r=
			\ds\int_0^{\infty}\Big|\frac{\partial}{\partial z}I(z,r)_{|z=1}\Big|r^{n-1} \mathrm{d} r\\&=&\frac{1}{p}
			\int_0^{\infty}\left|n-(p')^\frac{1}{1-p}r^{p'}\right|e^{-(p')^{-p'} r^{p^\prime}}r^{n-1}\mathrm{d} r=:\tilde C_0; 
		\end{eqnarray*}
		it is clear that $\tilde C_0$ is strictly positive (and  finite),  which proves \eqref{z-sharp} in the case when $x_0= 0.$

		\textit{Case 2:} $x_0\neq 0$. By using Lebesgue's dominated converge theorem, it follows that  
		$$\lim_{z\to 1}\ds\int_{\R^n}\Big|e^{-(p')^{-p'} |x-x_0|^{p^\prime}}- \left(\frac{z}{p(1-\frac{z}{p'})}\right)^\frac{n}{p^2}  e^{-(p')^{-p'}z^\frac{1}{p-1} |x|^{p'}}\Big| \mathrm{d} x= $$$$=\ds\int_{\R^n}\Big|e^{-(p')^{-p'} |x-x_0|^{p^\prime}}-   e^{-(p')^{-p'} |x|^{p'}}\Big| \mathrm{d} x.$$ Since $x_0\neq 0$, we clearly have that the latter integral is strictly positive (and finite), which implies again \eqref{z-sharp}. \hfill $\square$\\
		
		Using Theorem \ref{main-theorem-HC} we can easily characterize the equality case  in the hypercontractivity estimate \eqref{HCI_ineq}, which was first established  in  Balogh, Don and Krist\'aly \cite{BDK} by  optimal mass transportation. 
		
		\begin{corollary}\label{HC=}
			Under the same assumptions as in Theorem  \ref{main-theorem-HC}, equality holds in \eqref{HCI_ineq} for some $g:\mathbb R^n\to \mathbb R$ and $t>0$ if and only if 
			\begin{equation} \label{HCI_extremizers-cor}
				g(x) = C - \left( \frac{\beta - \alpha}{\beta t} \right)^{p'-1}  \frac{\left| x - x_0 \right|^{p'}}{p'} \ \ {for\ a.e.}\ \  x\in \mathbb R^n,
			\end{equation}
			for some $x_0 \in \R^n$ and $C \in \R$.
		\end{corollary}
		
		\begin{proof}
			It is easy to check that if $g$ is given by \eqref{HCI_extremizers-cor}, then equality holds in \eqref{HCI_ineq}. Conversely, if equality holds in \eqref{HCI_ineq}, the hypercontractivity deficit vanishes, i.e., $
			\delta^{\sf HC}_{p ,t,\alpha,\beta}(g)=0.$ By the estimate \eqref{stab-HC}, it follows that for a.e. $x\in \mathbb R^n$, one has
			$$e^{-\theta \cdot \frac{|x-x_0|^{p^\prime}}{p^\prime}}- a^{-\frac{\alpha}{\beta}}  e^{\alpha g(x)}=0,$$ for some $a>0$ and $x_0\in \mathbb R^n,$ where $\theta=\alpha\left(\frac{\beta-\alpha}{\beta t}\right)^{p^{\prime}-1}$. By the latter relation we obtain that $g$ has the form given by \eqref{HCI_extremizers-cor}.
		\end{proof}



		\subsection{Proof of Theorem \ref{Theorem2}.} Given a function $g:\mathbb R^n\to \mathbb R$, we choose the functions   $u,v,w: \R^n \to \R_+$ as in the proof of Theorem \ref{main-theorem-HC}, see \eqref{definition-u-v-w}. Since $g$ is concave, then $x\mapsto w(x)=e^{\alpha g\left(\frac{\beta}{\alpha} x\right)}$ is log-concave. As in the proof of Theorem  \ref{main-theorem-HC}, see \eqref{condition_PL_1} and \eqref{eq_4},  relations \eqref{PL-feltetel} and \eqref{PL-epsilon} hold, respectively. 
		Therefore, we are in the position to apply Theorem \ref{Theorem_Boroczky-De} (and Remark \ref{remark-constant}), obtaining that  there exist a point $x_0 \in \R^n$ and a universal constant $c >0$ such that
		\begin{equation}\label{u-v-w-B_D-befor}
			\int_{\mathbb{R}^n}|a^{-\lambda}w(x) -   v(x-x_0)| \mathrm{d} x \leq c^n n^n \left({\frac{\delta^{\sf HC}_{p ,t,\alpha,\beta}(g)}{\tau}}\right)^\frac{1}{19} \int_{\mathbb{R}^n} v \mathrm{d} x,
		\end{equation} 
		where $a = \int_{\R^n} u / \int_{\R^n} v$,  $\lambda=\frac{\alpha}{\beta}$ and $\tau=\min \left(\frac{\alpha}{\beta}, 1-\frac{\alpha}{\beta}\right)$. Now, relation \eqref{eq_v} shows that the constant in \eqref{new-stability-Boroczky-De}
		can be chosen to be 
		\begin{equation}\label{C-n-p}
			C(n,p)=c^n n^n(p')^\frac{n}{p^\prime}\Gamma\left(\frac{n}{p^{\prime}} +1\right)  \omega_{n},
		\end{equation}
		which concludes the proof of \eqref{new-stability-Boroczky-De}.
		
		Now, we assume in addition that $g:\mathbb R^n\to \mathbb R$ is radially symmetric. We first prove that the function $x \to \mathbf{Q}_{t} g(x)  $ is also radially symmetric for every $t>0$. 
		Let us fix $t>0$.  The 
		radially symmetry of $g:\mathbb R^n\to \mathbb R$ means that it is invariant under the action of the orthogonal group $O(n),$ i.e., $g(\tau x)=g(x)$ for every $\tau\in O(n)$ and $x\in \mathbb R^n$. Furthermore, we recall that  $O(n)$ acts isometrically  in $\mathbb R^n$, i.e.,  $|\tau x|=|x|$ for  every $\tau\in O(n)$ and $x\in \mathbb R^n$, and also that $O(n)$ acts transitively on the unit sphere $\mathbb S^{n-1}$ of $\mathbb R^n$. These properties imply that for every $\tau\in O(n)$ and $x\in \mathbb R^n$, one has
		\begin{eqnarray*}
			\mathbf{Q}_{t} g (\tau x) &=&	\inf _{y \in \R^n}\left\{g(y)+\frac{|\tau x-y|^{p^{\prime}}}{p^{\prime} t^{p^{\prime}-1}}\right\}=\inf _{y \in \R^n}\left\{g(y)+\frac{| x-\tau^{-1}y|^{p^{\prime}}}{p^{\prime} t^{p^{\prime}-1}}\right\}\\&=&\inf _{\tilde y \in \R^n}\left\{g(\tau \tilde y)+\frac{| x-\tilde y|^{p^{\prime}}}{p^{\prime} t^{p^{\prime}-1}}\right\}=\inf _{\tilde y \in \R^n}\left\{g( \tilde y)+\frac{| x-\tilde y|^{p^{\prime}}}{p^{\prime} t^{p^{\prime}-1}}\right\}\\&=&\mathbf{Q}_{t} g (x),
		\end{eqnarray*} 
		which shows that $\mathbf{Q}_{t} g $ is  radially symmetric. 
		
		The latter property implies that the functions $u,v,w: \R^n \to \R_+$ introduced in \eqref{definition-u-v-w} are all  radially symmetric. Therefore, we may apply Theorem \ref{Theorem_Figalli_Ramos},  obtaining that  there exist a dimensional constant $ C(n) >0$ and a radially symmetric log-concave function $h: \R^n \to \R_+$  such that
		\begin{align*}
			\int_{\R^n}|v(x)- a^{-\lambda}w\left(x\right)| \mathrm{d} x &\leq
			\int_{\R^n}|v(x)- a^{-\lambda}h\left(x\right)| \mathrm{d} x +
			a^{-\lambda}\int_{\R^n}|w(x)- h\left(x\right)| \mathrm{d} x \\
			&\leq C(n) (2+\varepsilon) \left(\frac{\varepsilon}{\tau}\right)^{1 / 2} \int_{\mathbb{R}} v(x)  \mathrm{d} x, 
		\end{align*} 
		where $a = \int_{\R^n} u / \int_{\R^n} v$ and $\tau=\min \left(\frac{\alpha}{\beta}, 1-\frac{\alpha}{\beta}\right)$. 
		Again, due to relation \eqref{eq_v}, if we choose 
		\begin{equation*}\label{C-n-p-2}
			C(n,p)= 3C(n) (p')^\frac{n}{p^\prime}\Gamma\left(\frac{n}{p^{\prime}} +1\right)  \omega_{n},
		\end{equation*}
		we conclude  the proof of \eqref{new-stability-Figalli-Ramos-HC}. The sharpness of the exponent  $\frac{1}{2}$ in \eqref{new-stability-Figalli-Ramos-HC} can be deduced similarly as in the proof of Theorem  \ref{main-theorem-HC} (see Step 2 for $x_0=0.)$ 
		\hfill $\square$

		\section{From HC to LSI: proof of Theorem \ref{Theorem-LSI-main}}\label{section-4}
		
		Before providing the proof of Theorem \ref{Theorem-LSI-main}, we establish a close connection between the hypercontractivity and $L^p$-logarithmic Sobolev deficits, defined in \eqref{HCI_deficit} and \eqref{LSI_deficit}, respectively.

		\begin{proposition} \label{prop_1}
			Let $n \geq 1$, $p>1$  and $g: \mathbb{R}^n \to \mathbb{R}$ be a  locally Lipschitz function with the property that $e^{g/p}\in W^{1,p}(\mathbb R^n)$ and 
			\begin{equation}\label{g-growth-prop}
				g(x)\geq c_1-c_2|x|^{p'},\ x\in \mathbb R^n,
			\end{equation}
			some $c_1\in \mathbb R$ and $c_2>0$.
			Let us denote by 
			$$y = \frac{1}{n}{\|e^{g}\|_1^{-1}} \int_{\R^n} e^{  g}\left|\nabla g\right|^p \mathrm{d} x .$$
			Then it follows that $0 < y < \infty$ and the following holds:
			\begin{equation} \label{limit-HC-LSI}
				\lim_{t \rightarrow 0^+} \frac{\delta_{p, t, 1, 1+yt}^{\sf HC}(g)}{t} = {y}  \delta_p^{\sf LSI}\left(e^{  g/p}\right).
			\end{equation}
			
		\end{proposition}	
		\begin{proof} First of all, note that the condition $e^{  g/p}\in W^{1,p}(\mathbb R^n)$ implies that $0 < y < \infty$ by the definition of $y$. Moreover, since $g$ is locally Lipschitz and verifies \eqref{g-growth-prop}, 	by Proposition \ref{H-J-equation}  one has 
			for a.e.\ $x\in \mathbb R^n$ that 
			\begin{equation}\label{Bobkov-Ledoux-limit-0}
				\frac{d}{dt}{\bf Q}_tg(x)\Big|_{t=0}= -\frac{1}{p}|\nabla g(x)|^p.
			\end{equation}
			
In order to  prove relation \eqref{limit-HC-LSI},  in the hypercontractivity deficit $\delta_{p, t, \alpha, \beta}^{\sf HC}(g)$ we choose $\alpha:=1$ and $\beta:=\beta(t) = 1+yt$,	$t\geq 0$, for some $y>0$, that is going to be chosen later.
	 Then, we have that $1 <  \beta(t)$ for all $t>0$; in addition, let us consider the function $$F(t) = \left\|e^{\mathbf{Q}_{t} g}\right\|_{\beta(t)}, \ t\geq 0.$$ 
			In particular, $F(0)=\lim_{t\to 0^+}F(t)= \left\|e^g\right\|_1$. 
			Then, by \eqref{HCI_deficit}, for any $t>0$, we have that
			\begin{equation}\label{deficit-hyp-lsi-1}
				\delta^{\sf HC}_{p ,t,1,\beta(t)}(g) = \mathcal{C}_{p, t, 1, \beta(t)} \frac{
					F(0)}{F(t)} -1 .
			\end{equation}
			Using the explicit form of the $\mathcal{C}_{p, t, 1, \beta(t)}$ by its definition \eqref{HCI_constant}, by a direct computation we obtain
			\begin{equation}\label{limit-11}
				\lim_{t\to 0^+}\mathcal{C}_{p, t, 1, \beta(t)}=1,
			\end{equation} 
			which implies that
			\begin{equation}\label{def-to=0}
				\lim_{t\to 0^+}\delta^{\sf HC}_{p ,t,1,\beta(t)}(g)=0.
			\end{equation}
			In addition, both functions $t\mapsto \mathcal{C}_{p, t, 1, \beta(t)}$ and $t \to F(t)$ are differentiable at $0$. By using \eqref{Bobkov-Ledoux-limit-0} and the chain rule, one has that  
			$$
			\frac{F^{\prime}(0)}{F(0)} = \frac{y}{  \left\|e^{g}\right\|_{1}}\left(\Ent\left(\mathrm{e}^{   g}\right) -\frac{1}{py} \int_{\R^n} \left|\nabla g\right|^p \mathrm{e}^{ g} \mathrm{~d} x\right) .
			$$ 
			Therefore, by the L'H\^ospital rule and relations  \eqref{deficit-hyp-lsi-1} and \eqref{limit-11}, we have that
			for any $y>0$, 
			\begin{eqnarray} \label{choice_of_y}
				\lim_{t \rightarrow 0^+} \frac{\delta_{p, t, 1, 1+yt}^{\sf HC}(g)}{t}& =&\frac{d}{dt}\left(  \mathcal{C}_{p, t, 1, \beta(t)}  \frac{
					F(0) }{F(t) }\right)\Big|_{t=0}= \frac{d}{dt}\mathcal{C}_{p, t, 1, \beta(t)}\Big|_{t=0}- \frac{F^{\prime}(0)}{F(0)} \nonumber \\
				&=&
				{y}\left(
				\frac{n}{p}\log y-\log \left(\left(p^{\prime}\right)^{\frac{n}{p^{\prime}}} \Gamma\left(\frac{n}{p^{\prime}}+1\right) \omega_n\right) -n\right)\nonumber\\
				&&-\frac{y}{  \left\|e^{g}\right\|_{1}}\left(\Ent\left(\mathrm{e}^{ g}\right) -\frac{1}{py} \int_{\R^n} \left|\nabla g\right|^p \mathrm{e}^{ g} \mathrm{~d} x\right).
			\end{eqnarray}
			Choosing  
			$$y := \frac{1}{n}{\|e^{g}\|_1^{-1}} \int_{\R^n} e^{  g}\left|\nabla g\right|^p \mathrm{d} x $$
			in the latter relation, and using the definition of the $L^p$-logarithmic Sobolev deficit \eqref{LSI_deficit} for the function $f:=e^{  g/p}$,  we obtain exactly  \eqref{limit-HC-LSI}. 
		\end{proof}

		\begin{remark}\rm
			The choice of the value for $y>0$ is motivated by the following consideration. Notice that \eqref{choice_of_y} is equivalent to
			\begin{align*} 
				\frac{\Ent\left(|f|^p\right)}{\|f\|_p^p} 
				&= \frac{p^{p-1}}{y}  \frac{1}{\|f\|_p^p}  \int_{\R^n} \left| \nabla f \right|^p \, \mathrm{d}x 
				+ \frac{n}{p} \log y 
				 \nonumber \\
				&\quad - \log \left( \left( p' \right)^{\frac{n}{p'}} \Gamma\left( \frac{n}{p'} + 1 \right) \omega_n \right) - n
				- \frac{1}{y} \lim_{t \rightarrow 0^+} \frac{\delta_{p, t, 1, 1+yt}^{\sf HC}(g)}{t}
			\end{align*}
			for the function $f:=e^{  g/p}$.
			Analyzing the function 
			$$
			h(y) =  \frac{p^{p-1}}{y} \frac{1}{\|f\|_p^p} \int_{\R^n}\left|\nabla f\right|^p \mathrm{d} x  +\frac{n}{p} \log y,
			$$
			we observe that $h$ has a global minimum at
			\begin{equation} \label{y_min}
				y= \frac{p^p}{n}  \frac{1}{\|f\|_p^p} \int_{\R^n}\left|\nabla f\right|^p \mathrm{d} x = \frac{1}{n\|e^{g}\|_1} \int_{\R^n} e^{  g}\left|\nabla g\right|^p \mathrm{d} x . 
			\end{equation}
		\end{remark}

		\vspace{0.5cm}
		\noindent \textit{Proof of Theorem \ref{Theorem-LSI-main}.} We divide the proof into three steps. 
		
		\textbf{Step 1:} \textit{proof of} \eqref{lsi-estimate}.  Let $p>1$,  and  $f \in W^{1,p}(\mathbb R^n)$  be a log-concave function verifying \eqref{one-sided}; accordingly, let  $g: \mathbb{R}^n \to \mathbb{R}$ be a  concave function such that $f \coloneqq e^{ g/p}$, i.e., $g = p \log f$. Note that the function $g$ is also locally Lipschitz. 
		
		Let us consider the function $\beta: (0, \infty) \to \R$ given by $\beta(t) = yt+1$, where  $y>0$ is defined as in the statement of Proposition \ref{prop_1}. According to \eqref{def-to=0}, we have that 	$
		\delta^{\sf HC}_{p ,t,1,\beta(t)}(g)  \ll 1
		$ for $t>0$ sufficiently small. Thus, we may apply the first part of Theorem \ref{Theorem2}, see \eqref{new-stability-Boroczky-De},  obtaining that there exist
		$x_0 \in \mathbb{R}^n$ and the constant $C(n,p)>0$ from \eqref{C-n-p} such that    
		\begin{equation}\label{stability_from_Boroczky}
			\int_{\R^n}\left|e^{-\theta \cdot \frac{|x-x_0|^{p^\prime}}{p^\prime}}-a^{-\frac{1}{\beta(t)}} e^{  g\left( x\right)}\right| \mathrm{d} x  \leq C(n,p) \theta^{-\frac{n}{p'}} \left({\frac{\delta^{\sf HC}_{p ,t,1,\beta(t)}(g)}{\tau}}\right)^\frac{1}{19} ,
		\end{equation}
		where $\theta=\left(\frac{\beta(t)-1}{t\beta(t) }\right)^{p^{\prime}-1}$,  $a =  \int_{\R^n} e^{\beta \mathbf{Q}_{t} g(x)} \mathrm{d} x / \int_{\R^n} e^{-\theta \beta^{p'}(t) \cdot \frac{|x|^{p^\prime}}{p^\prime}} \mathrm{d} x $, and $\tau=\min \left(\frac{1}{\beta(t)}, 1-\frac{1}{\beta(t)}\right)$.
		
		We aim to study the quantities appearing in relation \eqref{stability_from_Boroczky} in the limit as $t \to 0$. First, we observe that if $t\to 0$, then 
		$$\beta(t)\to 1,\ \ \theta\to y^{p'-1} 
		,\ \ \tau = \min \left(\frac{1}{\beta(t)}, 1-\frac{1}{\beta(t)}\right) = 
		1-\frac{1}{yt+1} = \frac{y t}{yt+1}\to 0,$$
		and $$a\to \frac{\int_{\R^n} e^{   g} \mathrm{d} x }{ \int_{\R^n} e^{-  y^{p^{\prime}-1} \cdot \frac{|x|^{p^\prime}}{p^\prime}} \mathrm{d} x}.$$
		In addition, due to Proposition \ref{prop_1}, we have that 
		$$\lim_{t\to 0^+}\frac{\delta^{\sf HC}_{p ,t,1,\beta(t)}(g)}{\tau}= \lim_{t\to 0^+}\frac{\delta^{\sf HC}_{p ,t,1,\beta(t)}(g)}{t} \cdot\frac{yt+1}{y}= \delta_p^{\sf LSI}\left(e^{  g/p}\right).$$
		Therefore, by the above limits,  \eqref{integral-formula} and \eqref{C-n-p}, by letting $t \to 0$ in \eqref{stability_from_Boroczky}, it follows that
		\begin{equation*}
			\int_{\R^n}\left| \frac{y^{\frac{n}{p}} e^{-\frac{y^{p'-1}}{p^\prime}   |x-x_0|^{p^\prime}}}{\Gamma\left(\frac{n}{p^{\prime}}+1\right)\omega_n \left(p^{\prime}\right)^{\frac{n}{p^{\prime}}} 
			} - \frac{e^{   g\left( x\right)}}{\int_{\R^n} e^{  g}} \right| \mathrm{d} x  \leq c^n n^n \cdot \delta_p^{\sf LSI}(e^{  g/p})^\frac{1}{19}.
		\end{equation*}
		where
		$c>0$ is the absolute constant from Theorem \ref{Theorem_Boroczky-De} (see also Remark \ref{remark-constant}). 
		Substituting the value of $y$ from \eqref{y_min} and using $f^p = e^{  g}$, a straightforward  calculation yields the desired inequality \eqref{lsi-estimate}, i.e., 
		$$\int_{\R^n}\left| {\sf C}_2  e^{-\frac{p}{{\sf C}_1} |x-x_0|^{p^\prime}} - \frac{f^p(x)}{\|f\|_p^p} \right| \mathrm{d} x  \leq c^n n^n\cdot 
		\delta_p^{\sf LSI}(f)^\frac{1}{19},$$
		where
		$${\sf C}_1 =\frac{pp'}{y^{p'-1}} = p' \left( \frac{n}{p} \right)^{p^\prime-1} {\|f\|_p^{p^\prime}} \left( \int_{\mathbb{R}^n} |\nabla f|^p \,        \mathrm{d}x \right)^{1-p^\prime}$$
		and 
		$${\sf C}_2= \frac{y^{\frac{n}{p}} }{\Gamma\left(\frac{n}{p^{\prime}}+1\right)\omega_n \left(p^{\prime}\right)^{\frac{n}{p^{\prime}}} 
		}=\left(\left(\frac{{\sf C}_1}{p}\right)^{\frac{n}{p^\prime}} \Gamma\left(\frac{n}{p^\prime}+1\right) \omega_n\right)^{-1}.$$
		
		\textbf{Step 2:} \textit{proof of} \eqref{lsi-estimate-Schwarz}. We first state a result which will be useful later. Namely, let us assume  that  $\tilde f\in W^{1,p}(\mathbb R^n)$  is any radially symmetric,  log-concave function verifying \eqref{one-sided}.
		A similar limiting argument as in Step 1 -- by applying the  second part of Theorem \ref{Theorem2}, see \eqref{new-stability-Figalli-Ramos-HC} --  immediately implies that 
		\begin{equation}\label{lsi-estimate-radialisra}
			\int_{\R^n}\left| \tilde {\sf C}_2  e^{-\frac{p}{\tilde {\sf C}_1} |x|^{p^\prime}} - \frac{\tilde f^p(x)}{\|\tilde f\|_p^p} \right| \mathrm{d} x  \leq 3C(n)
			\delta_p^{\sf LSI}(\tilde f)^\frac{1}{2}, 
		\end{equation}
		where $\tilde {\sf C}_1>0$ and  $\tilde {\sf C}_2>0$ are the corresponding values from \eqref{C-1-C-2}, computed for $\tilde f$ instead of $f$, and $C(n)>0$  is the constant from Theorem  \ref{Theorem_Figalli_Ramos}. 
		
		Now, we focus on the proof of \eqref{lsi-estimate-Schwarz}. Let us fix a log-concave function $f: \mathbb{R}^n \to \mathbb{R}_+$  with $f \in W^{1,p}(\mathbb R^n)\setminus \{0\}$, and consider its  Schwarz-rearrangement  $f^\star \in W^{1,p}(\mathbb R^n)$, see e.g.\ Lieb and Loss \cite{Lieb-Loss}. By the P\'olya--Szeg\H o inequality and the layer cake representation, one has that 
		$$\int_{\mathbb R^n}|\nabla f|^p \geq \int_{\mathbb R^n}|\nabla f^\star|^p,\ \ \|f\|_p=\|f^\star\|_p\ \ {\rm and}\ \ \Ent(f^p)=\Ent((f^\star)^p).$$
		In particular, these relations imply that
		\begin{equation}\label{deficit-star}
			\delta_p^{\sf LSI}(f)\geq \delta_p^{\sf LSI}(f^\star). 
		\end{equation}
		By definition, $f^\star$ is radially symmetric and according to Proposition 24 in Milman and Rotem \cite{Milman-Rotem}, since $f$ is log-concave, $f^\star$ is also log-concave. In particular, applying inequality \eqref{lsi-estimate-radialisra} for $\tilde f:=f^\star$, combined with \eqref{deficit-star}, we obtain the desired inequality  \eqref{lsi-estimate-Schwarz}. 
		
		\textbf{Step 3:} \textit{sharpness of the exponent $\frac{1}{2}$ in \eqref{lsi-estimate-Schwarz}.}
		%
		%
		Let $n,p>1$. For convenience, we recall    
		\eqref{lsi-estimate-Schwarz}, namely, 
		\begin{equation}\label{1.12}
			\int_{\R^n}\left| \left(\left(\frac{C(n,p,f^\star)}{p}\right)^{\frac{n}{p^\prime}} \Gamma\left(\frac{n}{p^\prime}+1\right) \omega_n\right)^{-1}  e^{-\frac{p}{C(n,p,f^\star)} |x|^{p^\prime}} - \frac{(f^\star)^p(x)}{\|f\|_p^p} \right| \mathrm{d} x  \leq C(n) \delta_p^{\sf LSI}(f)^\frac{1}{2} ,
		\end{equation}
		where
		$$ C(n,p, f^\star) = p' \left( \frac{n}{p} \right)^{p^\prime-1} \|f\|_p^{p^\prime}\left( \int_{\mathbb{R}^n} |\nabla f^\star|^p \,        \mathrm{d}x \right)^{1-p^\prime}. $$
		
		At a first glance, as in the proof of the sharpness in Theorem \ref{main-theorem-HC},  we could try to use the family of functions $f_\epsilon=e^{g_\epsilon/p}$, $\epsilon>0$, where $g_\epsilon$ is from \eqref{g-epsilon}. However, it turns out that this choice is not appropriate, since $ \delta_p^{\sf LSI}(g_\epsilon)= \delta_p^{\sf LSI}(g_0)=0$ for all $\epsilon>0$. 
		Therefore, some other perturbation of the extremal function is needed. 
		
		For  $0\leq \varepsilon<p'-1$, let us consider the family of radially symmetric and log-concave Gaussian-type functions 
		$$f_\varepsilon(x)=e^{-\frac{1}{p}|x|^{p'-\varepsilon}},\ \ x\in \mathbb R^n.$$ Moreover, we also observe that  $f_\varepsilon \in W^{1,p}(\mathbb R^n)$ verifying the growth \eqref{one-sided}, 	
		and $f_\epsilon=f_\epsilon^\star$.
		
		First, we study the asymptotic behavior of the $L^p$-logarithmic Sobolev deficit 
		$\delta_p^{\sf LSI}(f_\varepsilon)$ as $\epsilon \to 0$. 
		A computation based on the integral formula \eqref{integral-formula} implies that  $$\|f_\varepsilon\|_p^p=\int_{\mathbb R}e^{-{|x|^{p'-\varepsilon}}}\mathrm{d} x =\frac{n\omega_n}{p'-\varepsilon}\Gamma\left(\frac{n}{p'-\varepsilon}\right),$$
		$$\int_{\mathbb R^n}f_\varepsilon^p\log f_\varepsilon^p \mathrm{d} x=-\int_{\mathbb R^n}|x|^{p'-\varepsilon}e^{-|x|^{p'-\varepsilon}}\mathrm{d} x=-\frac{n\omega_n}{p'-\varepsilon}\Gamma\left(1+\frac{n}{p'-\varepsilon}\right).$$
		Therefore, $$\Ent\left(f_\varepsilon^p \right)=-\frac{n\omega_n}{p'-\varepsilon}\Gamma\left(1+\frac{n}{p'-\varepsilon}\right)-\frac{n\omega_n}{p'-\varepsilon}\Gamma\left(\frac{n}{p'-\varepsilon}\right)\log\left(\frac{n\omega_n}{p'-\varepsilon}\Gamma\left(\frac{n}{p'-\varepsilon}\right)\right).$$
		In addition, since $$\nabla f_\varepsilon(x)=-\frac{p'-\varepsilon}{p}f_\varepsilon(x)|x|^{p'-\varepsilon-2} x,\ \  x\in \mathbb R^n \setminus \{0\},$$  it turns out that
		\begin{eqnarray*}
			\int_{\mathbb R^n}|\nabla f_\varepsilon|^p \mathrm{d} x &=&\left(\frac{p'-\varepsilon}{p}\right)^p\int_{\mathbb R^n} |x|^{p(p'-\varepsilon-1)}e^{-|x|^{p'-\varepsilon}} \mathrm{d} x\\&=&{n\omega_n}\frac{(p'-\varepsilon)^{p-1}}{p^p}\Gamma\left(\frac{p(p'-\varepsilon-1)+n}{p'-\varepsilon}\right).
		\end{eqnarray*}	
		Summing up the above expressions,  we obtain that 
		\begin{eqnarray*}
			\delta_p^{\sf LSI}\left(f_\varepsilon\right)&=&\frac{n}{p}\log\left(\mathcal{L}_{n,p}\left(\frac{p'-\varepsilon}{p}\right)^p\frac{\Gamma\left(\frac{p(p'-\varepsilon-1)+n}{p'-\varepsilon}\right)}{\Gamma\left(\frac{n}{p'-\varepsilon}\right)}\right)+\frac{n}{p'-\varepsilon}+\log\left(\frac{n\omega_n}{p'-\varepsilon}\Gamma\left(\frac{n}{p'-\varepsilon}\right)\right).
		\end{eqnarray*}
		Clearly we have that $$\delta_p^{\sf LSI}\left(f_0\right)=0,$$ which is expected due to the fact that $f_0=e^{-\frac{|\cdot|^{p'}}{p}}$ is an extremizer   in the $L^p$-Euclidean logarithmic Sobolev inequality
		\eqref{LSI_ineq}. To continue the proof we shall recall the definitions of the Digamma and Trigamma functions related to the Gamma function 
		$\Gamma:(0,\infty)\to \mathbb R$ given by
		$$\Gamma(x)=\int_0^\infty t^{x-1}e^{-t}{\rm d}t,\ \ x>0.$$
		The Digamma and Trigamma functions are defined as 
		$$\Psi(x)=\frac{d}{dx}\log \Gamma(x)=\frac{\Gamma'(x)}{\Gamma(x)}\ \ {\rm and}\ \ \Psi'(x),\ \ x>0,$$
		respectively. 
		
		A direct calculation  shows that 
		$$\frac{d}{d\varepsilon}\delta_p^{\sf LSI}\left(f_{ \varepsilon}\right)\Big|_{\varepsilon=0}=-\frac{n}{p'p}\left(\Psi\left(\frac{n}{p'}\right)\frac{n}{p'}-\Psi\left(\frac{n}{p'}+1\right)\frac{n}{p'}+1\right)=0,$$
		where we used the recurrence relation for the Digamma function  $\Psi(x+1)=\Psi(x)+\frac{1}{x},$ $x>0.$
		In fact, it follows that $\varepsilon\mapsto \delta_p^{\sf LSI}\left(f_\varepsilon\right)$ has the following second order behavior:
		\begin{equation}\label{LSI-eps-1}
			\delta_p^{\sf LSI}\left(f_\varepsilon\right)= K(n,p)\varepsilon^2+o(\varepsilon^2), \ \ 0<\varepsilon\ll 1,
		\end{equation}
		where 
		$$
		K(n,p)=\frac{2}{np'}\left(\frac{n^2}{(p')^2}+\frac{n}{p'}-p+1-\frac{n^2}{(p')^2}\left(\frac{n}{p'}-p+1\right)\Psi'(\frac{n}{p'})\right).
		$$
		Note that for every $p,n>1$; Proposition \ref{trigamma} below   with the choices  $x\coloneqq \frac{n}{p'}>0$ and $q\coloneqq p-1>0$ implies that  $$K(n,p)>0.$$

		
		Now, we are going to study the behavior of the left-hand side of \eqref{1.12}. First,  we have that 
		%
		%
		%
		%
		\begin{align*}
			C(n, p, f_\varepsilon) 
			&= p' \left( \frac{n}{p} \right)^{p'-1} \|f_\varepsilon\|_p^{p'} \left( \int_{\mathbb{R}^n} |\nabla f_\varepsilon|^p \, \mathrm{d}x \right)^{1 - p'} \\
			&= p' \, n^{p'-1}  \frac{p}{\left(p' - \varepsilon\right)^{p'}} 
			\left[  \frac{\Gamma\left( \frac{n}{p' - \varepsilon} \right)}{\Gamma\left( \frac{p(p' - \varepsilon - 1) + n}{p'- \varepsilon} \right)} \right]^{p'-1}.
		\end{align*}    
		According to the second order expansion of the deficit $\delta_p^{\sf LSI}\left(f_\varepsilon\right)$, see  \eqref{LSI-eps-1},	 the sharpness of the exponent $1/2$ in \eqref{lsi-estimate-Schwarz} follows once we prove that there exists $C=C(n,p)>0$ such that for every $\varepsilon>0$ small enough and every $x_0\in \mathbb R^n$, one has
		\begin{equation}\label{LSI-reduced}
			\int_{\R^n}\left| \left(\left(\frac{C(n,p,f_\varepsilon)}{p}\right)^{\frac{n}{p^\prime}} \Gamma\left(\frac{n}{p^\prime}+1\right) \omega_n\right)^{-1}  e^{-\frac{p}{C(n,p,f_\varepsilon)} |x|^{p^\prime}} - \frac{e^{-|x|^{p'-\varepsilon}}}{\|f_\varepsilon\|_p^p} \right| \mathrm{d} x  \geq C\varepsilon.
		\end{equation}    
		Since $$\|f_\varepsilon\|_p^p \to\omega_n \Gamma\left(\frac{n}{p'}+1\right)\ \ {\rm as}\ \ \varepsilon\to 0,$$
		relation \eqref{LSI-reduced} can be rewritten into the equivalent form (with eventually another constant $C>0$) 
		\begin{equation}\label{LSI-reduced-modified}
			\int_{\R^n}\left| \left(\frac{C(n,p,f_\varepsilon)}{p}\right)^{-\frac{n}{p^\prime}}  e^{-\frac{p}{C(n,p,f_\varepsilon)} |x|^{p^\prime}} -  e^{  -{|x|^{p'-\varepsilon}}}\right| \mathrm{d} x 	\geq C\varepsilon.
		\end{equation} 
		For simplicity, for every $\varepsilon\geq 0$ and $x\in \mathbb R^n$, let 
		$$g(\varepsilon,x)=\left(\frac{C(n,p,f_\varepsilon)}{p}\right)^{-\frac{n}{p^\prime}}  e^{-\frac{p}{C(n,p,f_\varepsilon)} |x|^{p^\prime}} \quad {\rm and} \quad h(\varepsilon,x)=e^{  -{|x|^{p'-\varepsilon}}}.$$ 
		Since
		$C(n, p, f_0) = p$,
		it follows that $g(0,x) - h(0,x) = 0$ for every $x\in \mathbb R^n$, while
		$$\frac{d}{d\varepsilon}g(\varepsilon,x)\Big|_{\varepsilon=0} =- 
		\frac{1}{n}\left(\frac{n^2}{(p')^2}-\left(\frac{n}{p'}\Psi(\frac{n}{p'})+\frac{n}{p'}+1\right)|x|^{p'}\right)e^{-|x|^{p'}}
		$$ and $$
		\frac{d}{d\varepsilon}h(\varepsilon,x)\Big|_{\varepsilon=0} =
		|x|^{p'}e^{-|x|^{p'}}\log |x|  .
		$$
		Therefore, Fatou's lemma, L'H\^ospital's rule and the above derivatives show that 
		\begin{eqnarray*}
			J&\coloneqq&\liminf_{\varepsilon\to 0}\frac{1}{\varepsilon}\int_{\mathbb R^n}\left|g(\varepsilon,x)-h(\varepsilon,x)\right|	\mathrm{d} x\\&\geq& 
			\int_{\mathbb R^n}\liminf_{\varepsilon\to 0}\frac{\left|g(\varepsilon,x)-h(\varepsilon,x)\right|}{\varepsilon}\mathrm{d} x\\&=&\frac{1}{n}\int_{\mathbb R^n}\left|\frac{n^2}{(p')^2}-\left(\frac{n}{p'}\Psi(\frac{n}{p'})+\frac{n}{p'}+1-n\log|x|\right)|x|^{p'}\right|e^{-|x|^{p'}}\mathrm{d} x.
		\end{eqnarray*}
		It is clear that the latter integral is finite and strictly positive, which ends the proof of \eqref{LSI-reduced-modified}.
		\hfill $\square$
		
	\begin{remark}\rm (Limit  case in Theorem \ref{Theorem-LSI-main} when $p\searrow 1$) \label{remark-Beckner} 
			In the following we intend to take the limit $p\searrow 1$ in \eqref{LSI_ineq}. To do that, we shall assume that there exists $p_0 >1$ such that for all $1< p < p_0$ the function $f$ satisfies the conditions of Theorem \ref{Theorem-LSI-main} and thus \eqref{LSI_ineq} holds. We denote by $\mathcal{F}_{p_0}$ this class of functions. Taking the limit $p \to 1$  we obtain the $L^1$-\textit{logarithmic Sobolev inequality}:  
				\begin{equation} \label{LSI_ineq=1}
				\frac{\Ent(|f|) }{\|f\|_1} 
				\leq {n} \log \left(  \frac{1}{n\omega_n^{1/n}} \frac{\ds\int_{\R^n} |\nabla f| \, \mathrm{d} x}{\|f\|_1} \right), \ \ \forall f\in \mathcal{F}_{p_0}.
			\end{equation}
		In fact, this inequality holds in the larger space of functions with bounded variation, ${\sf BV}(\mathbb R^n)$, rather than in $\mathcal{F}_{p_0}$, see Beckner \cite{Beckner} and Ledoux \cite{Ledoux}; moreover, equality holds in \eqref{LSI_ineq=1} -- studied on ${\sf BV}(\mathbb R^n)$ -- if and only if $f=\mathbbm{1}_{B(x_0,r)}$ for some $x_0\in \mathbb R^n$ and $r>0$, where $\mathbbm{1}_S$ is the characteristic function of the set $S\subset \mathbb R^n$ and $B(x_0,r)=\{x\in \mathbb R^n:|x-x_0|<r\}.$   
			
		As we pointed out in the Introduction, it is natural to ask the stability of \eqref{LSI_ineq=1} when we take the limit $p\searrow 1$ in Theorem \ref{Theorem-LSI-main}.  For $f\in \mathcal{F}_{p_0}$ by using the notations from \eqref{C-1-C-2}, one has that
		\begin{equation}\label{limit-C1}
			\lim_{p\searrow 1}{\sf C}_1^\frac{1}{p'}=n\frac{\|f\|_1}{\ds\int_{\mathbb{R}^n} |\nabla f| \,        \mathrm{d}x}=\frac{1}{y},
		\end{equation}
		where we use  the expression $y$ from \eqref{y_min} for $p=1.$
		Therefore, by the second term of \eqref{C-1-C-2}, we obtain that 
		$$\lim_{p\searrow 1}{\sf C}_2=y^n\omega_n^{-1}.$$
		Finally, on account of  \eqref{limit-C1}, let us observe that 
		$$\lim_{p\searrow 1}\frac{|x-x_0|^{p'}}{{\sf C}_1} =\lim_{p\searrow 1}\left(\frac{|x-x_0|}{{\sf C}_1^{1/p'}}\right)^{p'}=\begin{cases}
			0 , &\text{ if }\  |x-x_0|<1/y,\\
			+\infty, &\text{ if }\ |x-x_0|>1/y,
		\end{cases}$$
		thus for a.e.\ $x\in \mathbb R^n,$
		$$\lim_{p\searrow 1}e^{-p\frac{|x-x_0|^{p'}}{{\sf C}_1}} =\begin{cases}
			1 , &\text{ if }\  |x-x_0|<1/y,\\
			0, &\text{ if }\ |x-x_0|>1/y
		\end{cases}=\mathbbm{1}_{B(x_0,y^{-1})}(x).$$
		Having the above limits, it remains to take  $p\searrow 1$ in \eqref{lsi-estimate}, in order to obtain the stability result	\begin{equation}\label{lsi-estimate-limit}
			\int_{\R^n}\left| \frac{\mathbbm{1}_{B(x_0,y^{-1})}(x)}{\omega_n y^{-n}} - \frac{f(x)}{\|f\|_1} \right| \mathrm{d} x  \leq C(n)
			\delta_1^{\sf LSI}(f)^\frac{1}{19}; 
		\end{equation}
		the same can be done  in \eqref{lsi-estimate-Schwarz} in the radially symmetric case by replacing $1/19$ by $1/2$.
		
		It would be natural to extend this procedure and obtain a stability result for all function in the larger class of ${\sf BV}(\mathbb R^n)$ functions. Further comments on this problem is given in Section \ref{section-6}.
		
\end{remark}

		\section{Stability in Gaussian HC and Gaussian LSI}\label{section-5}
		
		\subsection{Stability in Gaussian hypercontractivity estimates} Let $$\mathrm{d} \mu(x)=(2\pi)^{-n/2}e^{-|x|^2/2}\mathrm{d} x$$
		be  the Gaussian measure, and denote by $\|\cdot\|_{\alpha,\mu}$ the $L^\alpha$-norm with respect to the measure $\mathrm{d} \mu.$ According to Bobkov, Gentil and Ledoux \cite[Therem 2.1]{BGL}, for every  $t>0$, $\alpha>0$ and function $g:\mathbb R^n\to \mathbb R$ with the property that
		$e^g \in L^{\alpha}(\R^n, \mathrm{d}\mu)$, one has  the \textit{Gaussian hypercontractivity estimate} 
		\begin{equation}\label{GHC}
			\|e^{{\bf Q}_t g}\|_{\alpha+t,\mu}\leq \|e^{g}\|_{\alpha,\mu}. 
		\end{equation}
		Moreover, the value $y=1$ in the norm $\|\cdot\|_{\alpha+yt,\mu}$ is optimal (it cannot be replaced by a smaller
		constant), while  equality holds in \eqref{GHC} if and only if for some $x_0\in \mathbb R^n$ and $C_0\in \mathbb R$, one has that  $g(x)=\langle x,x_0\rangle +C_0$, see Balogh and Krist\'aly \cite{BK-JEMS}. 
		
		According to \eqref{GHC}, we may define the \textit{Gaussian hypercontractivity deficit} of $g$, i.e., 
		\begin{equation}  \label{G-HC-D}
			\delta^{\sf GHC}_{\alpha ,t}(g) =  \frac{
				\|e^{g}\|_{\alpha,\mu}^\alpha}{	\|e^{{\bf Q}_t g}\|^\alpha_{\alpha+t,\mu}} -1  \geq 0.
		\end{equation}
		Our first result is a counterpart of Theorem \ref{main-theorem-HC}.

		\begin{theorem}\label{theorem-G-HC-1} Let $n \geq 1$, $\alpha>0$  and   $t>0$. Then, there exists a constant $ C(n,\alpha,t) >0$ such that for any  function $g:\mathbb R^n\to \mathbb R$ with $e^g \in L^{\alpha}(\R^n, \mathrm{d}\mu)$, 
			there exists $x_0\in \mathbb R^n$ such that
			\begin{equation}\label{G-stab-HC}
				\int_{\R^n}\Big|k e^{\langle x,x_0\rangle }-a^{-\frac{\alpha}{\alpha+t}} e^{\alpha g(x)} \Big| \mathrm{d}\mu (x)  \leq C(n,\alpha,t) \cdot   \delta^{\sf GHC}_{\alpha,t}(g)^\frac{1}{2} ,
			\end{equation}
			where 
			\begin{equation}\label{G-theta_and_a_definition}
				k=e^{-\frac{|x_0|^2}{2}}\ \  and\ \ a =  \int_{\R^n} e^{(\alpha+t) \mathbf{Q}_{t} g(x)} \mathrm{d} \mu(x) .
			\end{equation}
			Moreover, the exponent $\frac{1}{2}$ of $\delta^{\sf GHC}_{\alpha,t}(g)$ in \eqref{G-stab-HC} is sharp. 
		\end{theorem}
		
		\begin{proof} We divide the proof into two parts. 
			
			\textbf{Step 1:} \textit{proof of \eqref{G-stab-HC}}.
			We introduce  the functions  $u,v,w: \R^n \to \R_+$ given by 
			\begin{equation}\label{G-definition-u-v-w}
				u(x)=(2\pi)^{-n/2}e^{(\alpha+t) \mathbf{Q}_{t} g(x)-|x|^2/2}, \quad 
				v(x)=(2\pi)^{-n/2}e^{-|x|^2/2},  \quad
				w(x)=(2\pi)^{-n/2}e^{\alpha g\left( x\right)-|x|^2/2},
			\end{equation}
			and consider the constant $$\lambda=\frac{\alpha}{\alpha+t}\in (0,1). $$
			
			By the definition of ${\bf Q}_tg$, we have for every $x,y\in \mathbb R^n$ that
			$${\bf Q}_tg(x)\leq g(\lambda x+(1-\lambda)y)+\frac{t}{2(\alpha+t)^2}|x-y|^2.$$ This inequality and the identity $|\lambda x+(1-\lambda)y|^2=\lambda |x|^2+(1-\lambda)|y|^2-\frac{\alpha t}{(\alpha+t)^2}|x-y|^2$ imply that the functions from \eqref{G-definition-u-v-w} verify
			$$	u(x)^{\lambda}  v(y)^{1-\lambda} 
			\leq  w( \lambda x+(1-\lambda) y), $$
		and thus we can apply the Pr\'ekopa--Leindler inequality 
				\begin{equation} \label{ineq-PL} \ds\int_{\mathbb{R}^n} w\mathrm{d} x  \geq  \left(\int_{\mathbb{R}^n} u \mathrm{d} x \right)^\lambda \left(\int_{\mathbb{R}^n} v \mathrm{d} x \right)^{1-\lambda} .
				\end{equation}
			
			Furthermore, if $\varepsilon=	\delta^{\sf GHC}_{\alpha ,t}(g)$, relation \eqref{G-HC-D} can be equivalently written into the form 
			$$	\ds\int_{\mathbb{R}^n} w =(1+\varepsilon) \left(\int_{\mathbb{R}^n} u \right)^\lambda \left(\int_{\mathbb{R}^n} v \right)^{1-\lambda}.$$
			
			Therefore, we may apply Theorem \ref{Theorem_Figalli_etal} to  $u,v$ and $w$ from \eqref{G-definition-u-v-w}, obtaining that there exist $x_0\in \mathbb R^n$ and  a constant $C=C(n,
			\alpha,t)>0$ such that
			\begin{align*}
				\int_{\mathbb{R}^n}|a^{-\lambda}w(x) -   v(x-x_0)| \mathrm{d} x \leq C \sqrt{\varepsilon} \int_{\mathbb{R}^n} v \mathrm{d} x,
			\end{align*} 
			where 
			$$a =  \frac{\int_{\R^n} u}{\int_{\R^n} v} =\int_{\R^n} e^{(\alpha+t) \mathbf{Q}_{t} g(x)} \mathrm{d} \mu(x).
			$$
			It is clear that the latter inequality is equivalent to \eqref{G-stab-HC}. 
			
			\textbf{Step 2:} \textit{sharpness in \eqref{G-stab-HC}}. For simplicity, we consider $\alpha=t=1$ and we use the test-function $$g_\varepsilon(x)=-\varepsilon |x|^2,\ x\in \mathbb R^n,$$
			for $\varepsilon\in [0,\frac{1}{4})$. It is clear that $e^{g_\varepsilon} \in L^1(\R^n, \mathrm{d}\mu)$ and since ${\bf Q}_1g_\varepsilon(x)=-\frac{\varepsilon}{1-2\varepsilon}|x|^2,$ $x\in \mathbb R^n,$ a direct computation shows that
			$$ \|e^{g_\varepsilon}\|_{1,\mu}=(1+2\varepsilon)^{-\frac{n}{2}}\ \ {\rm and}\ \ \|e^{{\bf Q}_1 g_\varepsilon}\|_{2,\mu}=\left(\frac{1-2\varepsilon}{1+2\varepsilon}\right)^\frac{n}{4}.$$
			Therefore, it follows that 
			$$
			\delta^{\sf GHC}_{1,1}(g_\varepsilon) =  \frac{
				\|e^{g_\varepsilon}\|_{1,\mu}}{	\|e^{{\bf Q}_1 g_\varepsilon}\|_{2,\mu}} -1=(1-4\varepsilon^2)^{-\frac{n}{4}} -1, 
			$$
			thus 
			$$
			\delta^{\sf GHC}_{1,1}(g_\varepsilon) =  n\varepsilon^2+o(\varepsilon^2), \ \ 0<\varepsilon\ll 1.
			$$
			
			Due to the latter relation, the sharpness of the exponent $1/2$ in \eqref{G-stab-HC} follows once we prove that there exists $C>0$ such that for every $x_0\in \mathbb R$ and every small $\varepsilon>0$, one has 	
			\begin{equation}\label{-sharp-G-stab-HC}
				\int_{\R^n}\Big| e^{\langle x,x_0\rangle -\frac{|x_0|^2}{2}}-\left(\frac{1+2\varepsilon}{1-2\varepsilon}\right)^\frac{n}{4} e^{ g_\varepsilon(x)} \Big| \mathrm{d}\mu (x)  \geq C   \varepsilon.
			\end{equation}
			
			\textit{Case 1}: $x_0= 0$. In this case \eqref{-sharp-G-stab-HC} reduces to 
			$$	\int_{\R^n}\Big| 1-\left(\frac{1+2\varepsilon}{1-2\varepsilon}\right)^\frac{n}{4} e^{ -\varepsilon |x|^2} \Big| \mathrm{d}\mu (x)  \geq C   \varepsilon.$$
			By Fatou's lemma we have that
			\begin{eqnarray*}
				L&:=&\liminf_{\varepsilon\to 0}\frac{1}{\varepsilon}\int_{\R^n}\Big| 1-\left(\frac{1+2\varepsilon}{1-2\varepsilon}\right)^\frac{n}{4} e^{ -\varepsilon |x|^2} \Big| \mathrm{d}\mu (x)\\&\geq& \int_{\R^n}\liminf_{\varepsilon\to 0}\frac{1}{\varepsilon}\Big| 1-\left(\frac{1+2\varepsilon}{1-2\varepsilon}\right)^\frac{n}{4} e^{ -\varepsilon |x|^2} \Big| \mathrm{d}\mu (x)\\&=&\int_{\R^n}\left| -|x|^2+n \right| \mathrm{d}\mu (x)
			\end{eqnarray*}
			which is strictly positive and finite, proving  \eqref{-sharp-G-stab-HC} for $x_0=0$.

			\textit{Case 2}: $x_0\neq 0$. When $\varepsilon\to 0$, the left hand side of \eqref{-sharp-G-stab-HC}  becomes $$L:=\int_{\R^n}\Big| e^{\langle x,x_0\rangle -\frac{|x_0|^2}{2}}-1\Big|\mathrm{d}\mu (x),$$ which is finite and non-zero, proving  \eqref{-sharp-G-stab-HC} also for $x_0\neq 0$. 
		\end{proof}
		
		Similarly to Theorem \ref{Theorem2}, we can state the following stability results: 
%

			\begin{theorem} \label{G-Theorem2}
				Let $n \geq 1$. Then there exists a constant $ C(n) >0$ with the following property. Given any $t>0$, $\alpha>0$, and any concave function $g:\mathbb R^n\to \mathbb R$ with $e^g \in L^{\alpha}(\R^n, \mathrm{d}\mu)$ and 
				$
				\delta^{\sf GHC}_{\alpha,t}(g)  \ll 1,
				$  there exists $x_0 \in \mathbb{R}^n$ such that
				\begin{equation}\label{G-new-stability-Boroczky-De}
					\int_{\R^n}\Big|k e^{\langle x,x_0\rangle }-a^{-\frac{\alpha}{\alpha+t}} e^{\alpha g(x)} \Big| \mathrm{d}\mu (x)  \leq C(n) \cdot  \left( \frac{\delta^{\sf GHC}_{\alpha,t}(g)}{\tau}\right)^\frac{1}{19} ,
				\end{equation}
				where $k>0$ and
				$
				a>0
				$ 
				are from \eqref{G-theta_and_a_definition} and $\tau=\frac{\min(\alpha,t)}{\alpha+t}$.
				
				In addition to the above assumptions, if $g:\mathbb R^n\to \mathbb R$ is radially symmetric, then 
				\begin{equation}\label{G-new-stability-Figalli-Ramos-HC}
					\int_{\R^n}\Big|1-a^{-\frac{\alpha}{\alpha+t}} e^{\alpha g(x)} \Big| \mathrm{d}\mu (x)  \leq C(n) \cdot  \left( \frac{\delta^{\sf GHC}_{\alpha,t}(g)}{\tau}\right)^\frac{1}{2} ,
				\end{equation}
				and the exponent  $\frac{1}{2}$ of the latter term in \eqref{new-stability-Figalli-Ramos-HC} is sharp.
			\end{theorem}
			\begin{proof} It is a simple adaptation of the proof of Theorem  \ref{theorem-G-HC-1}, by using Theorems \ref{Theorem_Boroczky-De} and  \ref{Theorem_Figalli_Ramos}, respectively. For the sharpness of the exponent  $\frac{1}{2}$ we recall that the test-function  $g_\varepsilon(x)=-\varepsilon |x|^2,\ x\in \mathbb R^n,$ is both concave and radially symmetric, thus it belongs to the admissible set of functions. 
			\end{proof}
			
			\subsection{Stability in Gaussian logarithmic Sobolev inequality.}

			When $p=2$, the $L^2$-Euclidean logarithmic Sobolev inequality \eqref{LSI_ineq} is equivalent to the \textit{Gaussian logarithmic Sobolev inequality}, i.e., for every $f\in W^{1,2}(\mathbb R^n,\mathrm{d} \mu)\setminus\{0\}$  one has
			\begin{equation}\label{GLSI}
				\Ent_{\mu}(f^2)\leq 2 \int_{\mathbb R^n} |\nabla f|^2\mathrm{d} \mu,
			\end{equation}
			where 
			$\mathrm{d} \mu(x)=(2\pi)^{-n/2}e^{-|x|^2/2}\mathrm{d} x$
			is the Gaussian measure and 
			$$\Ent_{\mu}(f^2) = \int_{\R^n} f^2 \log f^2 \, \mathrm{d} \mu - \int_{\R^n} f^2 \mathrm{d} \mu \cdot\log \int_{\R^n} f^2 \mathrm{d} \mu.$$
			Moreover,  due to  Carlen \cite{Carlen}, equality holds in \eqref{GLSI} if and only if for some $c\in \mathbb R$ and $x_0\in \mathbb R^n,$ $$f(x)=c e^{\langle x, x_0\rangle},\ x\in \mathbb R^n.$$

			 To establish our result,
				we introduce the \textit{Gaussian logarithmic Sobolev deficit} for $f\in W^{1,2}(\mathbb R^n,\mathrm{d} \mu)\setminus\{0\}$ as
				\begin{equation} \label{GLSI_deficit}
					\ds\delta^{\sf GLSI}(f) =\frac{\ds 2 \int_{\mathbb R^n} |\nabla f|^2\mathrm{d} \mu-\Ent_{\mu}(f^2)}{\ds\int_{\R^n} f^2 \, \mathrm{d}\mu}  \geq 0.
				\end{equation}

			The stability in the Gaussian logarithmic Sobolev inequality reads as follows:

				\begin{theorem}\label{Gaussian-Theorem-LSI-main}
					Let $n \geq 1$. 	Then there exists a constant $ C(n) >0$ with the following property. Given any  log-concave function $f\in W^{1,2}(\mathbb R^n,\mathrm{d} \mu)$ satisfying the growth 
					condition  
					\begin{equation}\label{two-sided}
						f(x)\geq C_1e^{-C_2|x|^2},\ x\in \mathbb R^n,
					\end{equation}
					for some $C_1,C_2>0$, then there exists $x_0 \in \mathbb{R}^n$ such that
					\begin{equation}\label{rG-new-stability-Boroczky-De}
						\int_{\R^n}\Big| e^{\langle x,x_0\rangle -\frac{|x_0|^2}{2}}- \frac{f^2(x)}{\|f\|^2_{2,\mu}} \Big| \mathrm{d}\mu (x)  \leq C(n) \cdot   \delta^{\sf GLSI}(f)^\frac{1}{19}.
					\end{equation}
					
					In addition, if $f$ is radially symmetric, then  
					\begin{equation}\label{Gaussian-lsi-estimate-Schwarz}
						\int_{\R^n}\Big| 1- \frac{f^2(x)}{\|f\|^2_{2,\mu}} \Big| \mathrm{d}\mu (x)  \leq C(n) \cdot   \delta^{\sf GLSI}(f)^\frac{1}{2},
					\end{equation}
					and the exponent  $\frac{1}{2}$  in \eqref{Gaussian-lsi-estimate-Schwarz} is sharp.
				\end{theorem}

		\begin{proof}
			We divide the proof into two steps. 
			
			\textbf{Step 1:} \textit{proof of} \eqref{rG-new-stability-Boroczky-De} \& \eqref{Gaussian-lsi-estimate-Schwarz}. The proofs are based on Theorem \ref{G-Theorem2}, letting $t\to 0$ in \eqref{G-new-stability-Boroczky-De} and \eqref{G-new-stability-Figalli-Ramos-HC}, respectively. To do this, let $f\in W^{1,2}(\mathbb R^n,\mathrm{d} \mu)$ be a log-concave function  that verifies the  growth assumption \eqref{two-sided}. Let    $g: \mathbb{R}^n \to \mathbb{R}$ be a  concave function such that $f \coloneqq e^{ g/2}$, i.e., $g = 2 \log f$; in particular, since $f\in W^{1,2}(\mathbb R^n,\mathrm{d} \mu)$, we also have that $e^g\in L^1(\mathbb R^n,\mathrm{d}\mu).$ 
			
			First, a similar result as in \eqref{limit-HC-LSI} is needed, connecting the Gaussian hypercontractivity deficit to the Gaussian logarithmic Sobolev deficit. In fact,  we have:  
			\begin{equation} \label{G-limit-HC-LSI}
				\lim_{t \rightarrow 0^+} \frac{\delta^{\sf GHC}_{1,t}(g)}{t} =    \delta^{\sf GLSI}\left(e^{ g/2}\right).
			\end{equation}
			To prove this, let us introduce the  function $$F(t) = \left\|e^{\mathbf{Q}_{t} g}\right\|_{1 +t,\mu}, \ t\geq 0.$$ 
			Note that $F(0)=\lim_{t\to 0^+}F(t)= \left\|e^g\right\|_{1,\mu}<+\infty$ and $F$ is well-defined for every $t>0$. Moreover, 
			  by the definition of the Gaussian hypercontractivity estimate \eqref{G-HC-D}, one has for every $t>0$ that
			\begin{equation}\label{deficit-hyp-lsi}
				\delta^{\sf GHC}_{1,t}(g) = \frac{
					F(0) }{F(t) } -1 .
			\end{equation}
			It is clear that 
			\begin{equation}\label{G-def-to=0}
				\lim_{t\to 0^+}\delta^{\sf GHC}_{1,t}(g)=0,
			\end{equation}
			and 
			$$\lim_{t \rightarrow 0^+} \frac{\delta^{\sf GHC}_{1,t}(g)}{t} =-  \frac{F'(0)}{F(0)}.$$
			Thanks to the  growth assumption \eqref{two-sided}, and $g$ being locally Lipschitz on $\mathbb R^n$, one can apply 
			Proposition \ref{H-J-equation}, obtaining that 
			for a.e.\ $x\in \mathbb R^n$, 
			$$
			\frac{d}{dt}{\bf Q}_tg(x)\Big|_{t=0}= -\frac{1}{2}|\nabla g(x)|^2.
			$$
			Thus, the latter relation and the chain rule  give that 
			$$
			\frac{F^{\prime}(0)}{F(0)} =\frac{\rm d}{{\rm d}t}\Big|_{t=0}\log F(t)= \frac{1}{ \left\|e^{g}\right\|_{1}}\left(\Ent_\mu\left(\mathrm{e}^{ g}\right) -2 \int_{\R^n} \left|\nabla (e^{\frac{ g}{2}})\right|^2  \mathrm{~d} \mu\right) .
			$$ 
			Therefore,  we have that
			$$
			\lim_{t \rightarrow 0^+} \frac{\delta^{\sf GHC}_{1,t}(g)}{t} =-  \frac{F'(0)}{F(0)}=\frac{1}{  \left\|e^{g}\right\|_{1} }\left(2 \int_{\R^n} \left|\nabla (e^{\frac{  g}{2}})\right|^2  \mathrm{~d} \mu(x)-\Ent_\mu\left(\mathrm{e}^{  g}\right) \right)=   \delta^{\sf GLSI}\left(e^{  g/2}\right),
			$$
			which is precisely relation \eqref{G-limit-HC-LSI}. 
			
			Since $\tau=\frac{\min(1,t)}{1+t}$, relations  \eqref{rG-new-stability-Boroczky-De} and \eqref{Gaussian-lsi-estimate-Schwarz} follow  by \eqref{G-new-stability-Boroczky-De}, \eqref{G-new-stability-Figalli-Ramos-HC} and \eqref{G-limit-HC-LSI}, as $$\lim_{t\to 0}\frac{\delta^{\sf GHC}_{1,t}(g)}{\tau}=\lim_{t\to 0}\frac{\delta^{\sf GHC}_{1,t}(g)}{t}\cdot \frac{t}{\frac{\min(1,t)}{1+t}}=\delta^{\sf GLSI}\left(e^{  g/2}\right)=\delta^{\sf GLSI}\left(f\right).$$

			\textbf{Step 2:} \textit{sharpness in}  \eqref{Gaussian-lsi-estimate-Schwarz}. The proof is similar as in Theorem \ref{theorem-G-HC-1}; for completeness, we provide the details.    For  $\varepsilon\in (0,\frac{1}{4})$, we consider the function $$f_\varepsilon(x)=e^{-\frac{\varepsilon}{2}|x|^2},\ x\in \mathbb R^n.$$
			It is clear that $f_\varepsilon$ is log-concave, radially symmetric and verifies the  growth assumption \eqref{two-sided}; moreover, similar calculations as in the proof of Theorem \ref{theorem-G-HC-1} provide
			$$\|f_\varepsilon\|_{2,\mu}=(1+2\varepsilon)^{-\frac{n}{4}},\ \  \int_{\mathbb R^n} f_\varepsilon^2 \log f_\varepsilon^2 \mathrm{~d} \mu = -n\varepsilon (1+2\varepsilon)^{-\frac{n}{2}-1} $$
			and 
			$$ \int_{\mathbb R^n} |\nabla f_\varepsilon|^2\mathrm{~d} \mu = n\varepsilon^2(1+2\varepsilon)^{-\frac{n}{2}-1}. $$
			Therefore, 
			$$\ds\delta^{\sf GLSI}(f_\varepsilon) =\frac{\ds 2 \int_{\mathbb R^n} |\nabla f_\varepsilon|^2\mathrm{d} \mu-\Ent_{\mu}(f_\varepsilon^2)}{\ds\int_{\R^n} f_\varepsilon^2 \, \mathrm{d}\mu} =n\varepsilon-\frac{n}{2}\log(1+2\varepsilon),$$
			thus 
			$$
			\delta^{\sf GLSI}(f_\varepsilon)  =  n\varepsilon^2+o(\varepsilon^2), \ \ 0<\varepsilon\ll 1.
			$$
			
			According to this estimate, the optimality of the exponent $1/2$ in \eqref{Gaussian-lsi-estimate-Schwarz} follows once we prove that there exists $C>0$ such that for every small $\varepsilon>0$, one has 	
			\begin{equation}\label{-sharp-G-stab-LSI}
				\int_{\R^n}\Big| 1-(1+2\varepsilon)^{\frac{n}{2}}e^{-{\varepsilon}|x|^2}\Big| \mathrm{d}\mu (x)  \geq C   \varepsilon.
			\end{equation}
			By Fatou's lemma we have that
			\begin{eqnarray*}
				L&:=&\liminf_{\varepsilon\to 0}\frac{1}{\varepsilon}\int_{\R^n}\Big| 1-(1+2\varepsilon)^{\frac{n}{2}}e^{-{\varepsilon}|x|^2}\Big| \mathrm{d}\mu (x)\\&\geq& \int_{\R^n}\liminf_{\varepsilon\to 0}\frac{1}{\varepsilon}\Big| 1-(1+2\varepsilon)^{\frac{n}{2}}e^{-{\varepsilon}|x|^2}\Big| \mathrm{d}\mu (x)\\&=&\int_{\R^n}\left| -|x|^2+n \right| \mathrm{d}\mu (x)
			\end{eqnarray*}
			which is strictly positive and finite, proving  \eqref{-sharp-G-stab-LSI}. 
				\end{proof}

			\begin{remark}\rm 
					We notice that while the existing stability results  in the literature for the Gaussian logarithmic Sobolev inequality    refer to $L^2$-type estimates (see the aforementioned works  \cite{BDS},  \cite{Dolbeault-et-al}, \cite{Dolbeault-et-al_short},  \cite{FathiIndreiLedoux},  \cite{FeoIndreiPosteraroRoberto}), Theorem \ref{Gaussian-Theorem-LSI-main} is slightly weaker, being established in $L^1$-norms (see the left hand sides of \eqref{rG-new-stability-Boroczky-De} and \eqref{Gaussian-lsi-estimate-Schwarz}).  
			\end{remark}

	\section{Final comments}\label{section-6}
	
	\subsection{Optimal Pr\'ekopa--Leindler stability} Based on Theorems \ref{main-theorem-HC}  and \ref{Theorem2}, we believe that a general sharp hypercontractivity estimate can be established: given $n \geq 1$ and  $p>1$, then there exists a constant $ C(n,p) >0$ with the following property that for any $t>0$, $0<\alpha<\beta$, and any function $g:\mathbb R^n\to \mathbb R$ such that $e^g\in L^\alpha(\mathbb R^n)$, and
	$
	\delta^{\sf HC}_{p ,t,\alpha,\beta}(g)  \ll 1,
	$  there exists $x_0 \in \mathbb{R}^n$ such that
	$$	\int_{\R^n}\left|e^{-\theta \cdot \frac{|x|^{p^\prime}}{p^\prime}}-a^{-\frac{\alpha}{\beta}} e^{\alpha  g\left( x\right)}\right| \mathrm{d} x  \leq C(n,p) \theta^{-\frac{n}{p'}} \left({\frac{\delta^{\sf HC}_{p ,t,\alpha,\beta}(g)}{\tau}}\right)^\frac{1}{2},$$
	where  $\theta>0$ and
	$
	a>0
	$ 
	are from \eqref{theta_and_a_definition}, and $\tau=\min \left(\frac{\alpha}{\beta}, 1-\frac{\alpha}{\beta}\right)$. This result follows one we have a finer Pr\'ekopa--Leindler stability of the form: 
	
	\vspace{0.2cm}
\noindent 	{\bf Conjecture.} (Optimal Pr\'ekopa--Leindler stability) \textit{	For every $n \in \mathbb{N}$, there exists a dimensional  constant $ C(n) >0$ with the following property.
	Let $0<\lambda<1$, and $u, v, w: \mathbb{R}^n \rightarrow \mathbb{R}_{+}$ be functions such that  \eqref{PL-feltetel} and \eqref{PL-epsilon} hold for some $\varepsilon\in (0,1].$
	Then there exist a  $($log-concave$)$ function $h: \mathbb{R}^n \rightarrow \mathbb{R}_{+}$ such that
	$$
	\begin{aligned}
		\int_{\R^n}| u(x) - a^{1-\lambda}  h\left(x\right)| \mathrm{d} x & \leq C(n)\left(\frac{\varepsilon}{\tau}\right)^{1 / 2} \int_{\mathbb{R}} u(x) \mathrm{d} x, \\
		\int_{\R^n}|v(x)- a^{-\lambda}h\left(x\right)| \mathrm{d} x & \leq C(n)\left(\frac{\varepsilon}{\tau}\right)^{1 / 2} \int_{\mathbb{R}} v(x)  \mathrm{d} x, \\
		\int_{\mathbb{R}}|w(x)-h(x)|  \mathrm{d} x & \leq C(n)\left(\frac{\varepsilon}{\tau}\right)^{1 / 2} \int_{\R^n} w(x)  \mathrm{d} x,
	\end{aligned}
	$$
	where $a = \int_{\R^n} u / \int_{\R^n} v$ and $\tau= \min(\lambda, 1-\lambda)$.} 
\vspace{0.3cm}

\noindent 	 The Conjecture would be a stronger version of the main result of  Figalli, van Hintum  and Tiba \cite{FigallivanHintumTiba_25} and its validity is indicated in Remark 1.8 of their paper. Based on our method, this would imply a sharp version for the stability of the $L^p$-logarithmic Sobolev inequality, $p>1$.
	
	\subsection{Stability in $L^1$-logarithmic Sobolev inequality}
	In Remark \ref{remark-Beckner} we established a stability result for the $L^1$-logarithmic Sobolev inequality \eqref{LSI_ineq=1} under the restricted condition that $f\in \mathcal{F}_{p_0}$ for some $p_0>1$. Instead of the condition $f\in \mathcal{F}_{p_0}$ it would be  more natural to require that $f\in {\sf BV}(\mathbb R^n)$  and  the existence of $C>0$ and a non-zero measured set $S\subset \mathbb R^n$ such that $f(x)\geq C$ for a.e.\ $x\in S.$ 	However,  our limiting argument does not allow this general assumption to derive the stability \eqref{lsi-estimate-limit} in the $L^1$-logarithmic Sobolev inequality. Therefore, it would be interesting to find a direct way  -- without using the limit  procedure $(p\to 1)$ --  to prove \eqref{lsi-estimate-limit}, by adapting in a suitable manner  Proposition \ref{H-J-equation} (formally, for $p= 1$). This problem requires a careful analysis, of the Hopf--Lax semigroup given by \eqref{Q-t-definition} in the degenerate case $p=1$.

			\appendix
			\section{}
			
			\subsection{Hamilton--Jacobi equation at the origin} 
			In this subsection we focus on the validity of the Hamilton--Jacobi equation at the origin:
			\begin{equation}\label{Bobkov-Ledoux-limit-0-1}
				\frac{d}{dt}{\bf Q}_tg(x)\Big|_{t=0}= -\frac{1}{p}|\nabla g(x)|^p,
			\end{equation}
			under mild assumptions on the function $g:\mathbb R^n\to \mathbb R$. Note that \eqref{Bobkov-Ledoux-limit-0-1} is know to be valid when $g$ is bounded, see Bobkov and Ledoux \cite[Lemma A]{BL}; however, this class of functions is not enough for our purposes, see the proof of  Theorems  \ref{Theorem-LSI-main} and  \ref{Gaussian-Theorem-LSI-main}.
			
			\begin{proposition}\label{H-J-equation}
				Let $g:\mathbb R^n\to \mathbb R$ be a locally Lipschitz function such that for some $C_1\in \mathbb R$ and  $C_2>0$, one has
				\begin{equation}\label{growth-g}
					g(x)\geq C_1-C_2|x|^{p'},\ \ x\in \mathbb R^n.
				\end{equation} 
				Then  \eqref{Bobkov-Ledoux-limit-0-1} holds for a.e.\ $x\in \mathbb R^n.$ 
			\end{proposition}
			\begin{proof} By Bobkov and Ledoux \cite[Lemma A, p.\ 378]{BL}, we have for a.e.\ $x\in \mathbb R^n$ that
				$$	\frac{d}{dt}{\bf Q}_tg(x)\Big|_{t=0}\leq -\frac{1}{p}|\nabla g(x)|^p;$$ 
				note that this inequality does not require any restriction on $g$ (unless the a.e.\ differentiability, which follows by the locally Lipschitz character of $g$).

				For the opposite inequality, we  follow  the idea from  Balogh,   Krist\'aly and Tripaldi \cite[Proposition 4.1/(iv)]{BKT}. 
				Let $t_0:=(2p'C_2)^\frac{1}{1-p'}$, where $C_2>0$ is from \eqref{growth-g}. Then, for every compact set $K\subset \mathbb R^n$, one can prove that $(x,t)\mapsto {\bf Q}_tg(x)$ is 
				well-defined, uniformly bounded from below, and
				Lipschitz continuous on the set $[0,t_0]\times K$. Indeed, for $t=0$, the function $ {\bf Q}_ 0 g=g$ is Lipschitz continuous on $K$. Now, if $(t,x)\in (0,t_0]\times K$, one has that 
				\begin{eqnarray*}
					{\bf Q}_tg(x)&=&\inf _{y \in \R^n}\left\{g(y)+\frac{|x-y|^{p^{\prime}}}{p^{\prime} t^{p^{\prime}-1}}\right\}\geq C_1+\inf _{y \in \R^n}\left\{-C_2|y|^{p'}+\frac{|x-y|^{p^{\prime}}}{p^{\prime} t^{p^{\prime}-1}}\right\}\\&\geq & C_1+C_2\inf _{y \in \R^n}\left\{-|y|^{p'}+2|x-y|^{p^{\prime}}\right\}>-\infty,
				\end{eqnarray*}
				uniformly on $K$; here, we used the coercivity property  $\lim_{|y|\to \infty}(-|y|^{p'}+2|x-y|^{p^{\prime}})=+\infty$. In addition, it turns out that there exists $R_K>0$ such that for every  $(t,x)\in (0,t_0]\times K$, one has
				\begin{equation}\label{Q-t-bounded}
					{\bf Q}_tg(x)=\min _{|y| \leq R_K}\left\{g(y)+\frac{|x-y|^{p^{\prime}}}{p^{\prime} t^{p^{\prime}-1}}\right\}.
				\end{equation}

				Let us fix $x\in \mathbb R^n$ where $g$ is differentiable. By \eqref{Q-t-bounded}, there exists $R_x>0$ such that if $\{t_n\}_{n\mathbb N}\subset (0,t_0/2)$ is a sequence converging to $0,$ then one can find $y_n\in \mathbb R^n$ with  $|y_n|\leq R_x$ and 
				$${\bf Q}_{t_n}g(x)=g(y_n)+\frac{|x-y_n|^{p^{\prime}}}{p^{\prime} t_n^{p^{\prime}-1}}.$$
				It is clear that $\{y_n\}$ converges to $x.$ Indeed, if we assume  for a  subsequence of $\{y_n\}$, which is still denoted by $\{y_n\}$, that $|y_n-x|\geq l>0$ for every $n\in \mathbb N$, then the latter relation together with the fact that $\max_{|y|\leq R_x} g(y)<+\infty$, $t_n\to 0^+$ and  ${\bf Q}_{t_n}g(x)\leq g(x)$, provides a contradiction. 
				
				Since  $x\in \mathbb R^n$ is a differentiable point of $g$, for every $\varepsilon>0$, there exists $n_\varepsilon\in \mathbb N$ such that for every $n\geq n_\varepsilon$ one has
				$$\frac{g(y_n)-g(x)}{|y_n-x|}\geq -|\nabla g(x)|-\varepsilon.$$ 
				Therefore, for every $n\geq n_\varepsilon$, by  Young's inequality one has
				\begin{eqnarray*}
					\frac{{\bf Q}_{t_n}g(x)-g(x)}{t_n}&=&\frac{g(y_n)-g(x)}{t_n}+\frac{|x-y_n|^{p^{\prime}}}{p^{\prime} t_n^{p^{\prime}}}=\frac{g(y_n)-g(x)}{|y_n-x|}\cdot \frac{|y_n-x|}{t_n}+\frac{1}{p'}\left(\frac{|y_n-x|}{ t_n}\right)^{p'}\\&\geq & -(|\nabla g(x)|+\varepsilon)\cdot \frac{|y_n-x|}{t_n}+\frac{1}{p'}\left(\frac{|y_n-x|}{ t_n}\right)^{p'}\\&\geq &-\frac{1}{p}(|\nabla g(x)|+\varepsilon)^p,
				\end{eqnarray*}
				which implies that 
				$$	\frac{d}{dt}{\bf Q}_tg(x)\Big|_{t=0}\geq -\frac{1}{p}|\nabla g(x)|^p,$$
				concluding the proof of \eqref{Bobkov-Ledoux-limit-0-1}. 
			\end{proof}

			\subsection{A property of the Trigamma function }\label{subsection-2-2} We shall prove a property of the Trigamma function used in the proof of the sharpness of Theorem \ref{Theorem-LSI-main}.
			Recall the definition of the Gamma function $\Gamma:(0,\infty)\to \mathbb R$ by
			$$\Gamma(x)=\int_0^\infty t^{x-1}e^{-t}{\rm d}t,\ \ x>0, $$
			and the Digamma and Trigamma functions are defined as 
			$$\Psi(x)=\frac{d}{dx}\log \Gamma(x)=\frac{\Gamma'(x)}{\Gamma(x)}\ \ {\rm and}\ \ \Psi'(x),\ \ x>0,$$
			respectively.  
			
			\begin{proposition}\label{trigamma} For every $x,q>0$, one has that 
				$$
				x^2+x-q-x^2\left(x-q\right)\Psi'(x)>0. 
				$$
			\end{proposition}
			\begin{proof}
				For $x,q>0$, we define the function $h(x,q) \coloneqq x^2+x-q-x^2\left(x-q\right)\Psi'(x)$. Hence,
				$$\frac{\partial}{\partial q}h(x,q)=-1+x^2\Psi(1,x).$$
				By the series representation of the Trigamma function $$\Psi'(x)=\sum_{k=0}^\infty \frac{1}{(k+x)^2}$$ and the latter relation, for every $x>0$, it follows that
				$$\frac{\partial}{\partial q}h(x,q)=\sum_{k=1}^\infty \frac{x^2}{(k+x)^2}>0,$$
				i.e., the function $q\mapsto h(x,q)$ is increasing on $(0,\infty)$. In particular, for every $x,q>0$, we have that 
				\begin{equation}\label{h-estime}
					h(x,q)>\lim_{q\to 0^+}h(x,q)=x^2+x-x^3\Psi'(x).
				\end{equation}
				On the other hand, by Ismail and Muldoon \cite[Theorem 2.1]{Ismail-Muldoon}, we know that the function 
				$$j(x):= \Psi(x)-\log x+\frac{\alpha}{x},\ \ x>0,$$ is completely monotonic, i.e., $(-1)^k\frac{d^k}{dx^k}j(x)\geq 0$  for every $x\in (0,\infty)$ and $k \in \mathbb N$,  if and only if $\alpha\geq 1$.  In particular, for $k=1$ and $\alpha=1$, it follows that 
				$\Psi'(x)=\frac{d}{dx}\Psi(x)\leq \frac{1}{x}+\frac{1}{x^2}$ for every $x>0.$ Therefore, by \eqref{h-estime}, one has that $h(x,q)>0$ for every $x,q>0$, which ends the proof. 
			\end{proof}


\begin{thebibliography}{99}
				\bibitem{BallBoroczky1}
				K. M. Ball and K.\ J.  B\"or\"oczky, Stability of the Pr\'ekopa--Leindler inequality. \textit{Mathematika} 56 (2010), no. 2, 339--356.
				
				\bibitem{BallBoroczky2}
				K. M. Ball and K.\ J.  B\"or\"oczky, Stability of some versions of the Pr\'ekopa--Leindler inequality.
				\textit{Monatshefte f\"ur Mathematik}  163 (2011), no. 1, 1--14.
				
				\bibitem{BDK}  Z.\ M. Balogh, S. Don, and A. Krist\'aly,  Sharp weighted log-Sobolev inequalities: characterization of equality cases and applications. \textit{Trans. Amer. Math. Soc.} 377 (2024), no. 7, 5129--5163.
				
				\bibitem{BK-Advances} Z.\ M. Balogh and A. Krist\'aly,   Equality in Borell--Brascamp--Lieb inequalities on curved spaces. \textit{Adv. Math.} 339 (2018), 453--494.
				
				\bibitem{BK-JEMS} Z.\ M. Balogh and A. Krist\'aly, Logarithmic-Sobolev inequalities on non-compact Euclidean submanifolds: sharpness and rigidity. \textit{Preprint}, arXiv: https://arxiv.org/abs/2410.09419v1, 2024. 
				
				\bibitem{BKT} 	 Z.\ M. Balogh, A. Krist\'aly,  and F. Tripaldi,  Sharp log-Sobolev inequalities in ${\sf CD}(0,N)$ spaces with applications. \textit{J. Funct. Anal.} 286 (2024), no. 2, Paper No. 110217, 41 pp. 
				
				\bibitem{Beckner} W. Beckner, Geometric asymptotics and the logarithmic Sobolev inequality. \textit{Forum Math.} 11 (1999),
				105--137.
				
				\bibitem{BNT}  N. Bez, S. Nakamura, and H. Tsuji,  Stability of hypercontractivity, the logarithmic Sobolev inequality, and Talagrand's cost inequality. \textit{J. Funct. Anal.} 285 (2023), no. 10, Paper No. 110121, 66 pp. 
				
				\bibitem{BianchiEgnell} G. Bianchi and H. Egnell, A note on the Sobolev inequality. \textit{J. Funct. Anal.} 100 (1991), no. 1, 18--24.
				
				\bibitem{BDS} G. Brigati, J. Dolbeault, and N. Simonov,   Stability for the logarithmic Sobolev inequality. \textit{J. Funct. Anal.} 287 (2024), no. 8, Paper No. 110562, 21 pp.
				
				\bibitem{BGL} S. Bobkov, I. Gentil, and M. Ledoux, Hypercontractivity of Hamilton--Jacobi equations. \textit{J. Math. Pures Appl.} 80, 7 (2001), 669--696.
				
				\bibitem{BL}  S.G. Bobkov and M. Ledoux,  From Brunn--Minkowski to sharp Sobolev inequalities. \textit{Ann. Mat. Pura Appl.} (4) 187 (2008), no. 3, 369--384.
				
				\bibitem{BoroczkyDe}
				K.\ J.\ B\"or\"oczky and A. De, Stability of the Pr\'ekopa--Leindler inequality for log-concave functions. \textit{Adv. Math.} 386 (2021), 107810.
				
				\bibitem{Boroczky-Figalli-Ramos}
			K.\ J.\ B\"or\"oczky, A.  Figalli, and J.\ P.\ G.\ Ramos,  
				A quantitative stability result for the Pr\'ekopa--Leindler inequality for arbitrary measurable functions. 
				\textit{Ann. Inst. H. Poincar\'e C Anal. Non Lin\'eaire} 41 (2024), no. 3, 565--614. 
				
				\bibitem{BucurFragala} D. Bucur and I. Fragal\`a, Lower bounds for the Pr\'ekopa--Leindler deficit by some distances modulo translations. \textit{J. Convex Anal.}, 21 (2014), no. 1, 289--305.
				
				\bibitem{Carlen} E. A. Carlen, Superadditivity of Fisher’s information and logarithmic Sobolev inequalities. \textit{J. Funct. Anal.} 101 (1991), 194--211.
				
				\bibitem{dPD}	M. Del Pino and J. Dolbeault, The optimal Euclidean $L^p$-Sobolev logarithmic inequality. \textit{J. Funct. Anal.} 197 (2003),
				no. 1, 151--161.
				
				\bibitem{DSW} B. Deng, L. Sun, and J.-C. Wei,  Sharp quantitative estimates of Struwe's decomposition.\textit{ Duke Math. J.} 174 (2025), no. 1, 159--228. 
				
				\bibitem{Dolbeault-et-al} J. Dolbeault,  M.J. Esteban, A. Figalli, R.L. Frank, and M. Loss,  
				Sharp stability for Sobolev and log-Sobolev inequalities, with optimal dimensional dependence.  
				\textit{Camb. J. Math.} 13 (2025), no. 2, 359--430.
				
				\bibitem{Dolbeault-et-al_short} J. Dolbeault,  M.J. Esteban, A. Figalli, R.L. Frank, and M. Loss, A short review on Improvements and stability for some interpolation inequalities. \textit{Preprint,}  arXiv: https://arxiv.org/abs/2402.08527, 2024.
				
				
				\bibitem{Dubuc} S. Dubuc, Crit\`eres de convexit\'e et in\'egalit\'es int\'egrales. \textit{Ann. Inst. Fourier Grenoble} 27(1) (1977) 135--165.
				
				
				\bibitem{Evans} L.C. Evans, \textit{Partial Differential Equations}. American Mathematical Society, Providence, RI, 1998.
				\bibitem{FathiIndreiLedoux} M.  Fathi, E. Indrei, and M. Ledoux, Quantitative logarithmic Sobolev inequalities and stability estimates. \textit{Discrete and Continuous Dynamical Systems} 36 (2016), no. 12, 6835--6853.
				
				\bibitem{FeoIndreiPosteraroRoberto} F. Feo, E. Indrei, M.R. Posteraro, and C. Roberto, Some remarks on the stability of the log-Sobolev inequality for the Gaussian measure. \textit{Potential Anal.} 47 (2017), no. 1, 37--52. 
				
				\bibitem{FigalliJerison} A. Figalli and D. Jerison, Quantitative stability for the Brunn--Minkowski inequality. \textit{Adv. Math.} 314 (2017), 1--47.
				
				\bibitem{FigalliMaggiPratelli} A. Figalli, F. Maggi, and A. Pratelli, A mass transportation approach to quantitative isoperimetric inequalities. \textit{Invent. Math.} 182 (2010), 167--211.
				
				\bibitem{FigalliRamos} A. Figalli and J. P.G. Ramos, Improved stability versions of the Pr\'ekopa--Leindler inequality. \textit{Preprint}, arXiv:2410.01122,  \textit{Journal of Convex Analysis}, to appear, 2024.
				
				\bibitem{FigalliRu}	A. Figalli and  Y. R. Zhang, Sharp gradient stability for the Sobolev inequality.
				\textit{Duke Math. J.} 171 (2022), no. 12, 2407--2459.
				
				\bibitem{FigallivanHintumTiba_23} A. Figalli, P. van Hintum, and M. Tiba, Sharp quantitative stability of the Brunn--Minkowski inequality. \textit{Preprint}, arXiv:2310.20643, 2023.
				
				\bibitem{FigallivanHintumTiba_24} 
				A. Figalli, P. van Hintum, and M. Tiba, Sharp stability of the Brunn--Minkowski inequality via optimal mass transportation. \textit{Preprint}, arXiv:2407.10932, 2024.
				
				\bibitem{FigallivanHintumTiba_25} A. Figalli, P. van Hintum, and M. Tiba, Sharp quantitative stability for the Pr\'ekopa--Leindler and Borell--Brascamp--Lieb inequalities. \textit{Preprint}, 	arXiv:2501.04656, 2025.
				
				\bibitem{FuscoMaggiPratelli} N. Fusco, F. Maggi, and A. Pratelli, The sharp quantitative isoperimetric inequality. \textit{Ann. of Math.} (2), 168 (2008), no. 3, 941--980.
				
				\bibitem{Gentil_JFA} I. Gentil, The general optimal $L^p$-Euclidean logarithmic Sobolev inequality by Hamilton--Jacobi equations. \textit{J. Funct. Anal.} 202 (2003), 591--599.  
				
				\bibitem{Gentil_BullSci} I. Gentil, Ultracontractive bounds on Hamilton--Jacobi solutions. \textit{Bull. Sci. Math.} 126 (2002), 507--524.  
				
				%
				%
				
			\bibitem{Indrei-Marcon} E.	Indrei and D. Marcon,  
				A quantitative log-Sobolev inequality for a two parameter family of functions. 
				\textit{Int. Math. Res. Not. IMRN} 2014, no. 20, 5563--5580. 
				
				
				\bibitem{Ismail-Muldoon} M.E.H. Ismail and M.E. Muldoon, Inequalities and monotonicity properties for gamma and $q$-gamma functions, in: R.V.M. Zahar (Ed.), in: Approximation
				and Computation: A Festschrift in Honour of Walter Gautschi, ISNM, vol. 119, Birkhauser, Basel, 1994, pp. 309--323.
				
				\bibitem{Ledoux} M. Ledoux, Isoperimetry and Gaussian analysis, in: Lectures on Probability Theory and Statistiques,
				\'Ecole d'\'et\'e de probabilit\'es de St-Flour, 1994, in: Lecture Notes in Math., vol. 1648, Springer, Berlin,
				1996, pp. 165--294.
				
				\bibitem{Lieb-Loss} E.H. Lieb and M. Loss, Analysis. Second edition. Graduate Studies in Mathematics, 14. American Mathematical Society, Providence, RI, 2001.
				
				\bibitem{Milman-Rotem}  V. Milman and L. Rotem,   Mixed integrals and related inequalities. \textit{J. Funct. Anal.} 264 (2013), no. 2, 570--604.
				
				\bibitem{RossiSalani} A. Rossi and P. Salani, Stability for Borell--Brascamp--Lieb inequalities. \textit{Geometric Aspects of Functional Analysis: Israel Seminar (GAFA) 2014--2016}, Springer, 339--363, 2017.
				
				\bibitem{Takshi} T. Suguro, Stability of the logarithmic Sobolev inequality and uncertainty principle for the Tsallis entropy.\textit{ Nonlinear Anal.} 250 (2025), Paper No. 113644, 16 pp. 
				
			\end{thebibliography}
		\end{document}